\RequirePackage{fix-cm}
\documentclass[lewno11pt,smallextended,envcountsect,envcountsame]{svjour3} 

\smartqed  
\usepackage{graphicx}
\usepackage[english]{babel}

\usepackage{cite}
\usepackage{enumitem}
\usepackage{theoremref}
\usepackage{hyperref}
\usepackage{amsfonts}
\usepackage{amssymb}
\usepackage{amsmath}
\usepackage{dsfont,bbm}
\usepackage{enumitem}
\usepackage{multicol}

\usepackage{lineno}
\usepackage{color}%

\newcommand{\Intf}[1]{{\mathcal{I}}_{#1}}
\newcommand{\Expf}[1]{{E}_{#1}}
\newcommand{\m}[1]{{{m}}_{#1}}
\newcommand{\Dc}[1]{ \mathcal{D}_c (#1) }

\newcommand{\Tau}{\mathcal{T}}

\DeclareMathOperator{\epi}{epi}

\DeclareMathOperator{\dom}{dom}
\newcommand{\tto}{\rightrightarrows}
\newcommand{\R}{\mathbb{R}}

\newcommand{\Rex}{\overline{\mathbb{R}} }
\newcommand{\N}{\mathbb{N}}
\newcommand{\dmu}{d\mu}
\newcommand{\dnu}{d\nu}

\newcommand{\ACp}{\operatorname{AC}^{1,p}([a,b],X)}
\newcommand{\ACq}{\operatorname{AC}^{1,q}([a,b],X^{\ast})}
\newcommand{\eps}{\varepsilon}

\renewcommand{\H}{\mathcal{H}}
\newcommand{\1}{\mathds{1}}
 \newcommand{\Inte}{{\operatorname{Int}}}

\DeclareMathOperator{\supp}{supp}
\let\epsilon\varepsilon

\newcommand{\cl}[1]{\operatorname{cl}  #1 }
\usepackage{float}
\title{Integral functionals on nonseparable Banach spaces with  applications\thanks{J.G. Garrido and E. Vilches were supported by Centro de Modelamiento Matem\'atico (CMM), ACE210010 and FB210005, BASAL funds for center of excellence from ANID-Chile. P. P\'erez-Aros and E. Vilches were supported by ANID-Chile under grants Fondecyt regular N$^{\circ}$ 1200283 and   Fondecyt regular N$^{\circ}$ 1220886.}}
\begin{document}


\titlerunning{Integral functionals on nonseparable Banach spaces with  applications}
\author{Juan Guillermo Garrido \and Pedro P\'erez-Aros \and Emilio Vilches
}


\institute{Juan Guillermo Garrido \at
              Departamento de Ingeniería Matemática, Universidad de Chile, Santiago, Chile. \\
              \email{jgarrido@dim.uchile.cl}           
           \and
           Pedro P\'erez-Aros  \at
              Instituto de Ciencias de la Ingenier\'ia, Universidad de O'Higgins, Rancagua, Chile.\\ \email{pedro.perez@uoh.cl}
              \and
              Emilio Vilches \at
              Instituto de Ciencias de la Ingenier\'ia, Universidad de O'Higgins, Rancagua, Chile.\\ \email{emilio.vilches@uoh.cl}
}

\date{Received: date / Accepted: date}
\maketitle

\begin{abstract}
In this paper, we study integral functionals defined on spaces of functions with values on general (non-separable) Banach spaces. We introduce a new class of integrands and multifunctions for which we obtain measurable selection results. Then, we provide an interchange formula between integration and infimum, which enables us to get explicit formulas for the conjugate and Clarke subdifferential of integral functionals. Applications to expected functionals from stochastic programming, optimality conditions for a calculus of variation problem and sweeping processes are given.
\end{abstract}

\section{Introduction}

Let  $X$ be a Banach space and $(T,\mathcal{A},\mu)$ a measure space. In this paper, we study integral functionals of the form $\Intf{f}(x): =\int_T f_t(x(t))\dmu$, where $x\colon T\to X$ is a measurable function and $f \colon T\times X\to \Rex$ is an appropriate integrand. This class of functionals constitutes an important object for many areas of applied  mathematics, especially for convex analysis, where Rockafellar has extensively studied it for normal integrands in finite-dimensional Banach spaces (see, e.g., \cite{MR236689, MR310612}). Then, these results where extended to general separable Banach spaces  (see, e.g., \cite{MR0373611,MR0467310,levin1975convex,MR1485775}). The variational analysis of integral functionals typically relies on an appropriate formula for the interchange between integration and infimum, which can be obtained via a measurable selection for the unbounded epigraph multifunction (hence the importance of separability). We emphasize that measurable selection theorem for bounded multifunctions defined on non-separable spaces are available (e.g.  \cite{MR2642727}). However, the same does not happen for unbounded multifunctions (such as the epigraph), which prevents the subdifferential analysis of these operators, with the consequent lack of tools for the study of problems of the calculus of variations in non-separable Banach spaces.

Therefore, the aim of this paper is to study integral functionals defined on spaces of functions with values on general (not necessarily separable) Banach spaces.  To achieve this goal, we focus our attention on $\sigma$-finite \emph{inner regular} measures  defined over a complete $\sigma$-algebra. This setting covers the most important measure spaces (particularly any interval with the Lebesgue measure) and is enough for most applications.   Hence, we introduce a new class of integrands and multifunctions for which we can obtain measurable selection results. Then, we provide an interchange formula between integration and infimum, which enables us to get explicit formulas for the conjugate and Clarke subdifferential of integral functionals. Applications to expected functionals, the calculus of variations and sweeping processes are given.

The paper is organized as follows.  After some preliminaries, in Section \ref{sect-3}, we provide an interchange formula between integral and infimum over continuous functions when the epigraph of the integrand is convex and lower semicontinuous. Section \ref{sect-4} is dedicated to studying the class Lusin integrands.  Section \ref{sect-5} introduces the notions of admissible integrands and admissible multifunctions, which enables us to find measurable selection results for general Banach spaces. In Section \ref{sect-6}, we provide the main result of this paper, which is the conjugate formula for an admissible integrand over decomposable spaces.  Then, we give calculus rules for the  $\epsilon$-subdifferencial of integral functionals defined over the space of $p$-integrable functions, distinguishing between the case  $p<\infty$ and $p=\infty$. In Section \ref{sect-7}, we characterize the Clarke subdifferential of integral functionals on general Banach spaces, which enables us, in Section \ref{sect-8}, to obtain  optimality conditions for a calculus of variations problem (see, e.g., \cite{MR1488695}).  Finally, in Section \ref{sect-9}, we prove that two formulations of the sweeping process are equivalent, using the interchange formula of Section \ref{sect-3}. 

\section{Notation and preliminary results}
In what follows,  $(X,\| \cdot\| )$ will be a Banach space endowed with the norm $\Vert \cdot \Vert$. The topological dual of $X$ is denoted by $X^\ast$ and the duality product between $X$ and $X^{\ast}$ is denoted by $\langle \cdot,\cdot \rangle$. 
The weak$^\ast$ topology on $X^\ast$ is simply denoted by $w^\ast$.  For a set $A$ of a  topological space $(Z,\tau)$, we denote its interior and closure by $\Inte_{\tau} A $ and $\cl^{\tau} A$, respectively. When no confusion arises,  we simply omit the symbol $\tau$.  The closed ball of radius $r$  and center $x$ is denoted by $\mathbb{B}_r(x)$.

As usual, we denote by $\Gamma_0(X)$ the  set of proper,  convex and lower semicontinuous defined on $X$ with values in the extended real line $\Rex:= [-\infty,\infty ]$. For $C\subset X$, we denote by $\delta_C$ as its indicator function, which is equal to $0$ over $C$ and $+\infty$ over $X\setminus C$. For  $g\colon X\to \Rex$  its conjugate function is $g^\ast \colon X^\ast \to \Rex$ given by  $g^\ast(x^\ast):=\sup\left\{ \langle x^\ast, x\rangle - g(x) \colon x\in X  \right\}$. If $C\subset X$, we define the support function of $C$ as $\sigma_C := (\delta_C)^\ast$. For a point $x\in X$, with $|g(x)|<+\infty$ and $\epsilon \geq 0$, we define the $\epsilon$-subdifferential of $g$ at $x$, by
\begin{align*}
\partial_{\epsilon} g(x) :=\left\{  x^\ast \in X^\ast \colon \langle x^\ast , y- x\rangle \leq g(y) - g(x)+ \epsilon,\text{ for all } y \in X  \right\}.
\end{align*}
for convenience we set $\partial_{\epsilon} g(x) = \emptyset$ for $x \in X$ with $|g(x)|=+\infty$. For $\epsilon=0$, the above definition recovers the subdiferential of convex analysis. For $C\subset X$ and $\epsilon\geq 0$, we define $N^\epsilon_C(x) := \partial_\epsilon \delta_C(x)$ for $x\in C$, which corresponds to $\epsilon$-normal cone of $C$.
The \emph{Fr\'echet subdifferential} of $g$ at $x$ is defined by $$\hat \partial g(x)=\left\{ x^\ast\in X^\ast\colon \liminf_{\|h\|\to 0} \frac{g(x+h)-g(x)-\langle x^\ast,h\rangle}{\|h\|}\geq 0 \right\}.$$ 
The \emph{limiting/Mordukhovich subdifferential} of $g$ at $x$ is defined by 
$$
\partial g(x):=\{w^{\ast}-\lim x_n^\ast \colon x_n^\ast\in\hat \partial g(x_n) \textrm{ and } (x_n,g(x_n))\to (x,g(x)) \}.
$$

If $C\subset X$, we define $N_C(x) := \partial \delta_C(x)$ for $x\in C$. If $C$ is a convex set, then $N_C(x) = N_C^0(x)$. Let $g\colon X\to\R$ locally Lipschitz around $x\in X$. The \emph{Clarke subdifferential} of $g$ at $x$ is defined by
$$\overline{\partial} g(x) := \{ x^\ast\in X\colon  \langle x^\ast,v\rangle \leq dg(x)(v) \textrm{ for all } v\in X \},$$
where 
$dg(x)(v) := \limsup_{y\to x, \ t\to 0^+}t^{-1}\left[g(y+tv)-g(y)\right]$.  We refer to \cite{MR2144010,MR1491362} for more details.



Given $(T,d)$ a metric space, by $\mathcal{B}(T)$ we denote the Borel $\sigma$-algebra on $T$.  The set of all continuous function $\phi \colon T\to X$ is denoted by $\mathcal{C}(T, X)$. 

 A multifunction $\Phi \colon T \tto X$ is said to be  \emph{lower semicontinous} (lsc)  at $t \in T$ if,  for every open set $V$ of $X$ such that $\Phi(t)\cap V \neq \emptyset$, there exists a neighbourhood $U$ of $t$ such that $\Phi(s) \cap V \neq \emptyset$ for all $s\in U$. Moreover, $\Phi$ is said lsc if it is lsc at every point of $\dom \Phi:=\{t\in T\colon \Phi(t)\neq \emptyset\}$.

\noindent In the sequel, $(T, \mathcal{A},\mu)$ will denote a complete $\sigma$-finite measure space containing $\mathcal{B}(T)$.  A function $s\colon T\to X$ is said \emph{simple}, if there are  disjoint sets  $A_1,...,A_n\in\mathcal{A} $ and $x_1,...,x_n\in X$ such that $s(t) = \sum_{k=1}^n x_k\1_{A_k}(t)$.

We say that $\phi\colon T \to X$ is \emph{Borel measurable} if $\phi^{-1}(U) \in\mathcal{A}$ for any open $U\subset X$. Furthermore, $\phi$ is said \emph{strongly measurable} if it is  the pointwise limit of simple functions. Finally,  $\phi $ is said \emph{weakly measurable} if for all $x^\ast \in X^\ast$ the map $t\mapsto \langle x^\ast,\phi(t)\rangle$ is  measurable. It is well-known that the above notions of measurability coincide when the space $X$ is separable.

\noindent According to \cite{MR3409135,MR1491362}, the  \emph{upper integral} of a measurable function $g\colon T \to \R$ is defined by
\begin{align*}
	\int_T g(t) \dmu := \int_T \max\{ g(t),0\} \dmu   +\int_T \min\{ g(t),0\} \dmu,
\end{align*}
with the convention $+\infty+(-\infty) =+\infty$. Given $f\colon T \times X \to \Rex$ and $t\in T$,  we denote by $f_t$ the function $x\mapsto f(t,x)$.  The epigraph multifunction (associated to $f$) is the set-valued map $t\tto \epi f_t:=\{ (x,\alpha) \in X\times \R \colon f_t(x)	\leq \alpha	\}$.

For $f\colon X\to \Rex$ we define the \emph{lower closure} of $f$ as as the unique function $\cl(f)\colon X\to\Rex$  for which   $\epi \cl(f)= \cl(\epi f)$.

A measure $\mu$ on $\mathcal{A} $ is said \emph{inner regular}  if for all $\varepsilon>0$ and $B \in \mathcal{A}$ of finite measure, there exists a compact set $K \subset  B$ such that $\mu(B\setminus K) < \epsilon$. We refer to \cite[Definition 2.15]{MR924157} for more details.  Throughout this paper we assume that $\mu$ is inner regular.  

The following result is a Lusin's theorem for strongly measurable functions in general (nonseparable) spaces.
\begin{proposition}\label{lusin}
    Let $X$ be a Banach space, $\mu$ an inner regular measure and $f \colon  T \to X$ a strongly measurable function. Then, for all $\epsilon>0$ and $C\in \mathcal{A}$ of finite measure, there exists a compact set $K\subset C$ such that $\mu(C\setminus K) < \epsilon$ and $f\colon K\to X$ is continuous.
\end{proposition}

\begin{proof} Fix $\varepsilon>0$ and $C\in \mathcal{A}$ of finite measure. 
   By virtue of \cite[Ch.2 Theorem 2]{MR0453964}, there exists $B\subset C$ with $\mu(C\setminus B)=0$ such that $S=f(B)$ is separable. Let us define $N := C\setminus B$ and set $Y$ as the closed subspace generated by $S$. It is clear that $Y$ is a separable Banach space. It follows that $f\colon B\to Y$ is measurable.  Hence, by  \cite[Theorem 7.1.13]{MR2267655},  we obtain a compact set $K\subset B$ such that $\mu(B\setminus K)<\epsilon$ and $f\colon K\to Y$ is continuous. Since $\mu(N)=0$, we obtain that $\mu(C\setminus K)<\epsilon$, which ends the proof.

\end{proof}

 For $p\in [1,\infty)$, the space of $p$-integrable functions with the usual norm is denoted by $L^p(T,X)$. The space of  essentially bounded functions with the usual norm is denoted by $L^\infty(T,X)$.  When $X=\R$, we simply write $L^p(T)$ instead of $L^p(T,X)$.
  
An element $\lambda\in L^\infty(T,X)^\ast$ is said \textit{singular measure with respect to $\mu$} if there exists an increasing sequence of measurable sets $(A_n)_{n\in \N}$ of $T$ such that $T = \bigcup_{n\in\N} A_n$ and for all $n\in \mathbb{N}$
$$
\lambda|_{A_n\cap A} = 0 \textrm{ for all } A\in\mathcal{A} \textrm{ with } \mu(A)<\infty,
$$
where $\lambda|_{B}$ is defined as $\lambda|_{B}(x) := \langle \lambda,x\1_{B}\rangle$ for $x\in L^\infty(T,X)$. 

  
 We denote by $\ACp$ the space of absolutely continuous functions whose a.e. derivative belongs to $L^p([a,b],X)$. It is well known that if $f\in \ACp$ then $$f(t)=f(a)+\int_a^t\dot{f}(s)ds \textrm{ for all } t\in [a,b],$$
 where $\dot{f}$ denotes the a.e. derivate of $f\in \ACp$. It is well-known that $\ACp$ endowed with the norm $\|f\|_{\ACp} = |f(a)| + \|\dot{f}\|_{p}$ is a Banach space.
 
\section{Interchange of infimum and integral on continuous functions}\label{sect-3}
In this section, we provide a result for the interchange of infimum and integral over the set of $X$-valued continous functions. Throughout this section, $(T,d)$ is a metric space and $(T,\mathcal{A},\mu)$ is a $\sigma$-finite measure space with inner regular measure $\mu$.

\begin{lemma}\label{first_lemma}
	Let $\Tau_{X^\ast}$ be a topology on $X^\ast$ such that the duality product is continuous with respect to the product topology $\Tau_X\times \Tau_{X^\ast}$. Let $f\colon T \times X\to \Rex$ be a function such that $t \tto \epi f_t$ is lsc.  Then, the following assertions hold. 
	\begin{enumerate}[label=(\alph*),ref=\ref{first_lemma}-(\alph*)]
	    \item \label{f_l_a} The map $(t,x^\ast) \to f^\ast_t(x^\ast)$ is lsc with respect to $\Tau_T\times \Tau_{X^\ast}$. 
	    \item \label{f_l_b} The infimal value function $t\mapsto \m{f}(t):=\inf_{x\in X}f_t(x)$ is usc. 
	    \item  \label{f_l_c} If  the functions $f_t$ are lsc for all $t\in T$, then for all $\phi \in \mathcal{C}(T,X)$ the function $t \mapsto  f_t(\phi(t))$ is $\mathcal{A}$-measurable.
	\end{enumerate}
\end{lemma}
\begin{proof}
	\noindent \textit{(a)} Let $t\in T$, $x^\ast \in X^\ast$  and $\alpha \in \mathbb{R}$  with $f_{t}^\ast(x^\ast) > \alpha$. By definition of the conjugate, there exists $x\in X$ such that $\langle x^\ast ,x \rangle - f_{t}(x) >\alpha$. Hence, by continuity of the duality product, there exist  $W_1 \in \mathcal{N}_{x}$, $W_2 \in \mathcal{N}_{x^\ast}$  and $\eta >0$ such that 
	\begin{align}\label{first_lemma_eq01}
		\langle w^\ast, w\rangle - \alpha >  \eta> f_{t}(x) \textrm{ for all } w\in W_1, w^\ast \in W_2.
	\end{align}
	Thus, $(x,f_{t}(x)) \in \epi f_{t} \cap\left(  W_1 \times (-\infty, \eta) \right)$.  By the lsc of $t\tto \epi f_t$, there exists $U \in \mathcal{N}_t$ such that $\epi f_{s} \cap\left(  W_1 \times (-\infty, \eta) \right)\neq \emptyset$  for all $s \in U$. Hence, by virtue of \eqref{first_lemma_eq01},  for all $s \in U$ there exists $w_{s}  \in W_1$ such that
	\begin{align*}
		f_{ s}(w_{s}) < \eta < \langle w^\ast, w_{s}\rangle - \alpha \textrm{ for all }  w^\ast \in W_2.
	\end{align*}
	Consequently, $f^\ast_{s}(w^\ast) > \alpha$ for all $s\in U$ and $w^\ast \in W_2$, which proves \textit{(a)}. 
	\newline \noindent \textit{(b)} We have that $m_f(t) = -f_t^\ast(0)$ and by Lemma \ref{f_l_a}, we get that $m_f$ is usc.
	
	\noindent \textit{(c)} For each $k \in \N$, let us consider the function
	$$
	g_k(t,x)=f_t(x) \textrm{ if } \| x-\phi(t)\|<1/k;\,\, g_k(t,x)=+\infty \textrm{ if } \| x-\phi(t)\|\geq 1/k.
	$$
	We claim that $t \tto \epi g_k(t,\cdot)$ is lsc for all $k\in \N $. Indeed, consider $t \in T$, $(x, \alpha) \in X\times \R$ and  an open set  $U\times V$ such that $(x,\alpha) \in \epi g_k(t, \cdot) \cap\left( U\times V \right)$. Therefore, $\|x- \phi(t)\|  < 1/k$. Moreover, by continuity we can take open sets $U_1 \in \mathcal{N}_x$ and $W \in\mathcal{N}_t$ with $U_1 \subset U$ such that $\|x'- \phi(t')\| < 1/k$ for all $(t',x') \in W\times U_1$. In particular,  $(x,\alpha) \in  \epi f_t \cap U_1 \times V $, so by the lsc of the epigraph multifunction there exists $W_1 \in \mathcal{N}_t$ with $W_1 \subset W$ such that $\epi f_{t'} \cap U_1\times V \neq \emptyset$. Therefore, $\epi g_k(t', \cdot) \cap U\times V \neq \emptyset$ for all $t\in W_1$. \newline
	 Now, $t \mapsto \inf_{ x \in X} g_{k} (t,x) =-g^\ast_k(t,0_{X^\ast})$ is usc (by the previous part), hence, measurable. Finally, since $f_t $ is lsc  we have $f_t(\phi(t)) = \sup_{ k \in \N } \inf_{ x \in X} g_{k} (t,x)$, which proves the measurability of $t \mapsto f_t(\phi(t))$.
	
\end{proof}
The next result provides a continuous selection result for the sublevel sets of integrands defined over $T\times X$. Recall that the infimal value function (associated to $f$) is the function $\m{f}\colon T \to \Rex$ defined by $ \m{f}(t):=\inf_{ x \in X} f_t(x).$ 
\begin{lemma}\label{lemma_selection} 
      Let $f\colon T\times X\to \overline{\R}$ and $B\subset T$ be such that $B \ni t\tto \epi f_t$ is lsc with closed and convex  values. If there is an lsc function $\alpha\colon B\to\Rex$ such that $m_f(t)<\alpha(t)$ for all $t\in B$, then there exists a continuous function $x\colon B\to X$ such that $f_t(x(t))\leq \alpha(t)$ for all $t\in B$. 
\end{lemma}
\begin{proof}
    Define $G\colon B \tto X$ as $G(t):=\{ x\in X\colon f_t(x)\leq \alpha(t)\}$. Since, $\m{f}(t)< \alpha(t)$ for all $t\in B$, it follows that $G$ takes closed, convex and nonempty values. Now, we check that $G$ is lsc. Indeed, consider ${\bar{ t} } \in B$ and an open set $V \subset X$ such that $G({\bar{ t} }) \cap V \neq \emptyset$, so let $x \in V$,  such that $f_{{\bar{ t} }}(x)\leq \alpha({\bar{ t} })$.  Since $\m{f} ({\bar{ t} } ) < \alpha({\bar{ t} })$, there exists $z \in X$ such that 
	$f_{{\bar{ t} }}(z) < \alpha(t)$. Consider $z_\lambda := (1-\lambda)x + \lambda z$, since $z_\lambda \to x$ as $\lambda \to 0$, we have that there exists $\lambda_0 >0$ such that $z_\lambda \in V$ for all $\lambda \in (0,\lambda_0)$, furthermore, by convexity, $f_{t }(z_\lambda) <  \alpha(t)$. We fix $\lambda\in (0,\lambda_0)$ and $\eta\in (f_{t }(z_\lambda), \alpha(t))$. On the one hand, since $\alpha$ is lsc, the map $(t,\gamma)\in B\times \R\mapsto \alpha(t) - \gamma$ is lsc, then there exists $U_1\times I\in \mathcal{N}_{(t,\eta)}$ such that $\alpha(t)>\gamma$ for all $(t,\gamma)\in U_1\times I$. On the other hand, since $\epi f_t \cap \hat{V}  \neq \emptyset$, where $\hat V:= V \times  I $, there exists $U_2 \in\mathcal{N}_t $ such that $\epi f_t \cap \hat V\neq \emptyset$, for all $t\in U_2$, therefore $G(t)   \cap V \neq \emptyset $   for all $t \in U_1\cap U_2$,    which proves that $G$ is lsc. Hence, by Michael selection's theorem (see, e.g., \cite[Theorem 1.2]{MR203702}, see also \cite{MR2990543,MR3028182} for further generalizations), there exists $x \in \mathcal{C}(B,X)$ with  $x(t) \in G(t)$  for all $t\in B$, so $f_t(x(t))\leq \alpha(t), \forall t\in B$.
	
\end{proof}

For $f\colon T\times X\to \Rex$, we denote by $\Dc{\m{f}}$ the set of all integrable functions $h\in \mathcal{C}(T,\R)$ such that $m_f (t) < h(t)$ for all $t\in T$. We recall that if $T$ is a metric space (see, e.g., \cite[Chap VIII, Sec. 5, Theorem 4.3]{MR0478089}), then for any usc function $g\colon T \to \R$ there exists a continuous function $h\colon  T\to \R$ such that $g(t) < h(t)$, for all $t\in T$. 
However, it is not clear when the set $ \Dc{\m{f}}$ should be nonempty for a general measure space. Nevertheless, when this set is nonempty, we can establish that the value of the integral of $\m{f}$ can be expressed as the infimum of the integral of the functions in $ \Dc{\m{f}}$.

\begin{lemma}\label{second_lemma}
      Let $f\colon T\times X\to\Rex$ be a function such that $t\tto \epi f_t$ is lsc. If $\Dc{\m{f}}\neq\emptyset$, then the following equality holds:
      \begin{align*}
		\int_T \m{f}(t)\dmu =\inf\left\{ \int_T \beta(t)\dmu \colon \beta \in \Dc{\m{f}}	\right\}.
	\end{align*}
\end{lemma}

\begin{proof}
    First, let us consider the nonempty set $\mathcal{F}$ of all continuous functions $\ell  \colon T \to \R$ such that $\m{f}(t) < \ell(t)$ for all $t \in T$ (see, e.g.,  \cite[Chapter IX, \S 6, Proposition 5]{MR1726872}). Then,   by \cite[Theorem 4.7.1]{MR2267655}, there exists a sequence of continuous functions  $ (\ell_k) \subset \mathcal{F}$ such that $\hat \ell(t):= \inf_k \ell_k(t)$ satisfies that for all $\ell \in \mathcal{F}$, $\hat \ell(t) \leq \ell (t)$ for $\mu$-a.e. $t\in T$ (minimality). Hence,  by the Lebesgue dominate convergence theorem we have that
	\begin{align*}
		\int_T \hat \ell = \inf\left\{ \int_T \beta(t)\dmu \colon \beta \in \Dc{\m{f}}	\right\}.
	\end{align*}
	Now, we assume by contradiction that there exist a measurable set $A$  with positive measure such that $\m{f}(t) <  \hat \ell  (t)$ for all $t\in A$. Then, there exists $\beta \in \mathbb{R}$ such that 
	$\mu(A_\beta) >0$, where  $A_\beta := \{ t \colon \m{f}(t) < \beta < \hat \ell(t)\}.$
	Since $\mu$ is $\sigma$-finite and inner regular, we can find a compact set $K\subset A_\beta$ such that $\mu(K)>0$. Now, for each $t \in K$,  by virtue of \cite[Chapter IX,\S~6, Proposition 5]{MR1726872}, there exists a function $\rho_{t} \in \mathcal{F}$ such that $  \rho_t(t) < \beta$, and by continuity of $\rho_t$ there exists   $U_{t} \in \mathcal{N}_{t}$ such that $\rho_t(t') < \beta$ for all  $t'\in U_{t}.$ 
	So, $(U_t)_{t\in K}$ is an open cover of $K$ and by compactness, there exist $t_1,...,t_n\in K$ such that $K\subset \bigcup_{i=1}^n U_{t_i}$. Then, there exists $i$ with  $\mu(K\cap U_{t_i})>0$ and therefore $\mu(A_\beta\cap U_{t_i})>0$. So, $\rho_{t_i}$ contradicts the minimality of $\hat \ell$, which ends the proof.
\end{proof}

The next theorem is the main result of this section. It provides an interchange formula between infimum and integration for continuous functions.
\begin{theorem}[Interchange of infimum and integral on continuous functions]\label{mainresult}
	Let  $f\colon T\times X \to \Rex$ be  such that  $t \tto \epi f_t $  is  lsc with closed and convex values. If $\Dc{\m{f}}\neq \emptyset$, then,
	\begin{align*}
		\inf_{\phi\in \mathcal{C}(T,X)}	 \int_{ T} f_t(\phi(t)) \dmu =\int_{ T}\inf_{x\in X}f_t(x)\dmu=\int_{ T}\m{f}(t) \dmu.
	\end{align*}
\end{theorem}
\begin{proof} 
	It is clear that the inequality $\geq $ always holds. To prove the opposite inequality,  let us observe that, if $ \int_T  \m{f}(t) \dmu =+\infty$, then the result is trivial.  Hence, we can assume that $ \int_T  \m{f}(t) \dmu <+\infty$. Let $\beta \in \Dc{\m{f}}$, by Lemma \ref{lemma_selection}, there exists a continuous function $x\colon T\to X$ such that $f_t(x(t))\leq \beta(t)$ for all $t\in T$. Then, 
	\begin{align*}
		\inf_{\phi \in \mathcal{C}(T,X)}\int_{ T} f_t(\phi(t)) \dmu  \leq \int_T f_t(x(t))\dmu \leq \int_T \beta(t) \dmu.
	\end{align*}
	Finally, by virtue of Lemma \ref{second_lemma}, we obtain the desired inequality.
\end{proof}

The above theorem depends on the non-emptiness of the set $\Dc{\m{f}}$. The next proposition  gives a sufficient condition to ensure this condition.
\begin{proposition}
     If $(T,d)$ is a separable metric space, and $\mu$ satisfies that for all $x\in T$, there exists $U\in \mathcal{N}_x$ such that $\mu(U)<\infty$, then, $\Dc{\m{f}}\neq \emptyset$.
\end{proposition}
\begin{proof} 
    We will demonstrate the assertion, using the following three claims:\\
{\bf Claim~1:} \emph{There exists an open cover $(U_n)_{n\in\N}$ of $T$ such that $\mu(\cl(U_n))<\infty$ and $U_n\subset U_{n+1}$.} \newline
\noindent \emph{Proof of Claim 1}: 
First of all, we notice that for all $x\in T$, there exists $O_x\in \mathcal{N}_x$ such that $\mu(O_x)<\infty$, as $x\in O_x$, there exists $r_x>0$ such that $V_x := \{y\in T: d(x,y)<r_x\}\subset \{x\in T:d(x,y)\leq r_x\}\subset O_x$. Then, we have that $(V_x)_{x\in T}$ is an open cover of $T$, therefore, by virtue of \cite[Theorem 2.3.17]{MR2161427}, there exists a countable subcover $(V_{x_i})_{i\in \N}$ of $T$. Finally, we can define $U_n := \bigcup_{i=1}^n V_{x_i}$, which satisfies that $\mu(\cl(U_n))\leq \sum_{i=1}^n\mu(\cl(V_{x_i}))<\infty$, which concludes the claim.

\noindent {\bf Claim~2:} \emph{There is a continuous and integrable function $\varphi\colon T\to (0,\infty)$.}\newline
\noindent \emph{Proof of Claim 2}: 
By the Claim 1, there exists an open cover $(U_n)_{n\in\N}$ of $T$ such that $\mu(U_n)<\infty$, for all $n\in \N$. Since $(U_{n})_{n\in\N}$ is an open cover of $T$, there exists a partition of unity $\{\varphi_{n}\colon T\to [0,1]\}$ subordinated to $(U_{n})$ such that $\sum_{n}\varphi_{n}(t) = 1, \forall t\in T$, $\operatorname{supp}(\varphi_{n})\subset U_{n}$ for all $n\in\N$, $\{\operatorname{supp}(\varphi_{n})\}$ is a locally finite cover and $\varphi_{n}$ is continuous for all $n\in\N$ (see \cite[Chap. VIII, Sec. 5, Theorem 4.2]{MR0478089}). Since $\mu(U_{n})<\infty$, we have that $\varphi_{n}\in L^1(T)$. Now, let $(\epsilon_i)_{i\in\N}$ be a  sequence of positive numbers  such that $\epsilon_i\int_T\varphi_i< 2^{-i}$ for all $i\in\N$. Hence, we define $\varphi := \sum_{i\in\N} \epsilon_i\varphi_i$, which is continuous, integrable and positive.

\noindent {\bf Claim~3:} \emph{Let $\phi\colon T\to \R$ an upper bounded function such that there exists an open set $V$, such that $\supp{\phi}\subset V$ and $\mu(\cl(V))<\infty$. Then $\Dc{\phi}\neq \emptyset$.}\newline
\noindent \emph{Proof of Claim 3}: We set $F := \operatorname{supp}(\phi)$. Then, using \cite[Urysohn's Lemma 4.15]{MR1681462}, there exists $\ell\colon T\to [0,1]$ continuous such that $\ell|_{F}=1$ and $\ell|_{V^c}=0$.  It is clear that $\operatorname{supp}(\ell)\subset \cl(V)$ and, since $\mu(\cl(V))<\infty$, $\ell\in L^1(T)$. We define $M := \sup_{t\in T}\phi(t)<\infty$, and we consider the function $\varphi$ obtained in the Claim 2. We have that the function $g(t):=M\cdot \ell(t) + \varphi(t)$ is continuous, integrable and $g>\phi$. Therefore,  $g\in \Dc{\phi}$.

Next, we consider the open cover $(U_i)_{i\in\N}$ considered in Claim 1. Since $\m{f}$ is usc, for all $n\in\N$, the set $V_n = \{t\in T\colon m_f(t)<n\}$ is an open set. Thus, we can define $O_n:= U_n\cap V_n$, and, since $(U_n)_{n\in\N}$ is a increasing sequence of sets,  $T = \bigcup_{n\in\N} O_n$.  Let us consider a partition of unity $\{\varphi_i\colon T\to [0,1]\}_{i\in\N}$ subordinated to $\{O_i\}_{i\in\N}$ with  $\sum_{i\in\N}\varphi_i(t) = 1$ for all $t\in T$, $\operatorname{supp}(\varphi_i)\subset O_i$ for all $i\in\N$. Hence,  $\{\operatorname{supp}(\varphi_i)\}_{i\in\N}$ is a locally finite covering of $T$ and $\varphi_i$ is continuous for all $i\in\N$. If $p := \max\{\m{f},0\}$, we define $\psi_i := p\cdot \varphi_i$ for all $i\in \N$. It is clear that $\supp{\psi_i}\subset \supp{\varphi_i}\subset O_i$ and $\psi_i(t)\leq i$ for all $t\in T$. Hence, $\psi_i$ is upper bounded,  $\supp{\psi_i}\subset U_i$ and $\mu(\cl{U_i})<\infty$. Now, by Claim 3,  $\Dc{\psi_i}\neq\emptyset$. Moreover, by the Lemma \ref{second_lemma}, for all $i\in\N$, there exists $\beta_i\in \Dc{\psi_i}$ such that $$\int_T \beta_i(t)\dmu\leq \int_T\psi_i(t)\dmu + 2^{-i}.$$ Define $\beta:=\sum_{i\in \N} \beta_i$ and, since $0\leq\psi_i<\beta_i$ and $\beta_i$ is continuous, $\beta$ is lsc. Then, we have that 
$$ 
\m{f}(t)\leq p(t) = p(t)\sum_{i\in \N} \varphi_i(t)=\sum_{i\in\N}\psi_i(t)<\beta(t) \textrm{ and }\sum_{i\in\N}\psi_i = p.
$$
Moreover, $\int_T \beta(t)\dmu\leq \int_T p(t)\dmu + 1<\infty$. Hence, $\beta$ is an integrable and lsc function such that $\m{f}<\beta$. Finally, by using \cite[Chap VIII, Sec. 5]{MR0478089}, there is $h\colon T\to\R$ continuous such that $\m{f}<h<\beta$, which implies that $h\in \Dc{\m{f}}$.
\end{proof}

\section{Lusin integrands and basic properties}\label{sect-4}

In this section, inspired by Lusin's theorem, we introduce the notion of Lusin integrand and multifunction. We provide several properties for this class of functions.  Moreover, we study the stability of this notion under typical operations  in optimization and control theory.

Let $(T,\mathcal{A},\mu)$ be a measurable space. We say that a function $f\colon T\times X\to \Rex$ is a \emph{Lusin integrand} if  for every $\varepsilon>0$ and  $A\in \mathcal{A}$ of finite measure, there exists a compact set $B\subset A$ with  $\mu(A \setminus B) < \epsilon$  such that  $B\ni t \tto \epi f_t$ is lsc. Furthermore, we say that $(f^i)_{i\in I}$ is a \emph{uniform Lusin family}  if for every $\varepsilon>0$ and  $A\in \mathcal{A}$ of finite measure, there exists a compact set $B\subset A$ with  $\mu(A \setminus B) < \epsilon$  such that  $B\ni t \tto \epi f_t$ is lsc, for all $i\in I$. 

The notion of Lusin integrand can be extended to the case of multifunctions. We say that $C\colon T\tto X$ is a \emph{Lusin multifunction} if for every $\epsilon>0$ and $A\in \mathcal{A}$ of finite measure, there exists a compact set $B\subset A$ with  $\mu(A \setminus  B) < \epsilon$  such that  $B\ni t \tto C(t)$ is lsc. Furthermore, we say that $(C_i)_{i\in I}$ is a \emph{uniform  family of Lusin multifunctions} if  for every every $\epsilon>0$ and $A\in \mathcal{A}$ of finite measure, there exists a compact set $B\subset A$ with  $\mu(A \setminus  B) < \epsilon$  such that  $B\ni t \tto C(t)$ is lsc, for all $i\in I$. 

From the above definitions, it follows that $(C_i)_{i\in I}$ is a uniform  family of Lusin multifunctions if and only if $\{(t,x)\mapsto \delta_{C_i(t)}(x)\}_{i\in I}$, $i\in I$ is a uniform Lusin family (see Proposition \ref{prop_un_lu} below).

The following proposition provides basic properties of Lusin integrands.
\begin{proposition}\label{lemma_lusin_int}
The following assertions hold:
\begin{enumerate}[label=(\alph*), ref=\ref{lemma_lusin_int}-(\alph*)]
    \item 
    Let $f,g \colon  T\times X \to \R\cup \{ +\infty\}$ be two Lusin integrands. Then, the inf-convolution  $ f\square g\colon (t,x)\mapsto \inf_{y\in X}\{ f(y) + g(x-y)\}$ is a Lusin integrand. 

    \item\label{lus_c}  Let $f\colon T\times X\to\Rex$ be a Lusin integrand, then $\cl(f)$ is a Lusin integrand. 
    
    \item 
    Let $f \colon T\times X\to \R$ be a simple function, i.e., $f_t(x) = \sum_{k=1}^\infty a_i(t) \varphi_i(x)$, for $a_i \colon T \to \R$ measurable and $\varphi_i \colon X \to \R$ continuous. If for every compact set $K\subset T$ and $x\in X$, the above series converges uniformly over $K$ (for $x$ fixed), then $f$ is a Lusin integrand.
    \end{enumerate}
  
\end{proposition}
 \begin{proof}
       \textit{(a)} We start the proof with a claim. \\
       {\bf Claim~1:} \emph{Let $h\colon T\times X\to\R\cup\{+\infty\}$ and $B\subset T$.  Then $B\ni t\tto \epi h_t$ is lsc if and only if  $B\ni t\tto \epi_s h_t:=\{ (x,\alpha)\colon h(t,x)<\alpha \}$ is lsc.} \newline
\noindent \emph{Proof of Claim 1}: Let $U\times I$ be an open set of $X\times \R$. It is clear that 
$$
\{ t\in B\colon\epi g_t\cap(U\times I)\neq\emptyset \} = \{ t\in B\colon\epi_s g_t\cap(U\times I)\neq\emptyset \},
$$ which proves the claim.\\ 
        Now, let $\epsilon >0$ and $A\in \mathcal{A}$ of finite measure. Let $B_1,B_2\subset A$ two compact sets $B_1,B_2\subset A$ such that $\mu(A\setminus B_1)<\epsilon/2$ and $\mu(A\setminus B_2)<\epsilon/2$ so that $B_1\ni t\tto \epi f_t$ and $B_2\ni t\tto \epi g_t$ are lsc. If $B:=B_1\cap B_2$, then we have that $\mu(A\setminus B)<\epsilon$. Moreover, it is known that $ \epi_s (f\square g)_t  = \epi_s f_t + \epi_s g_t$. Hence, by virtue of Claim 1, we have that $B\in t\tto \epi_s f_t$ and $B\ni t\tto \epi_s g_t$ are lsc.  Moreover, due to \cite[Proposition 2.59]{MR1485775}, the map $B\ni t\tto \epi_s (f\square g)_t$ is lsc. Finally, again by Claim 1, we conclude that $B\ni t\tto \epi (f\square g)_t$ is lsc.
     
     \textit{(b)} By definition of $\cl(f)$, we have that $\epi \cl(f) = \cl(\epi f)$. Moreover,  for any open set $V \subset X\times \R$
     \begin{equation*}
         \{t \colon\epi f_t\cap V \neq\emptyset\} = \{t\colon\cl(\epi f_t)\cap V\neq\emptyset\}
         = \{t\colon\epi \cl(f_t)\cap V\neq\emptyset\},
     \end{equation*}
     which proves that $\cl f$ is a Lusin integrand.
     
     \textit{(c)} Let $\epsilon>0$ and $A\in \mathcal{A}$ of finite measure. By Proposition \ref{lusin}, for all $n\in\mathbb{N}$, there exists a compact set $K_n\subset A$ with  $\mu(A\setminus K_n)<\epsilon \cdot 2^{-n}$ such that $a_n\colon K_n\to\R$ is continuous. We set $K=\bigcap_{n\in \N}K_n$ for which $\mu(A\setminus K)<\epsilon$. We proceed to prove that $K\ni t\tto \epi f_t$ is lsc. Indeed, let $U\times I$ be an open set of $X\times \R$ and take $\hat{t}\in K$ such that there exists $(\hat{x},\hat{\alpha})\in U\times I$ and $f_{\hat{t}}(\hat{x})\leq \hat{\alpha}$. By assumption, the map  $K\ni t\mapsto f_t(\hat{x})$ is continuous. Since $I$ is open, there exists $\beta\in I$ such that $\beta>\hat{\alpha}$ and, by continuity, we can find a neighborhood $V$ of $\hat{t}$ such that $f_t(\hat{x})<\beta$ for all $t\in V$, which shows that $\{ t\in T\colon\epi f_t\cap (U\times I)\neq \emptyset \}$ is an open set.
    \end{proof}
    
   Let $(T,\mathcal{A},\mu)$ be a measurable space. We say that $F \colon T \times X \tto Y$ is a \emph{Scorza-Dragoni}  map if for every $\epsilon>0$ and $A\in \mathcal{A}$ of finite measure there exists a compact set $B\subset A$ with $\mu(A \setminus  B) < \epsilon$ and  such that  $F \colon B \times X \to Y$  is lsc with closed graph. In particular,  we say that $f \colon T \times X \to Y$ is a  \emph{Scorza-Dragoni function}  if the set-valued mapping $F(t,x):=\{ f_t(x)\}$ is a  Scorza-Dragoni  map, i.e., for every $\epsilon>0$ and   $A\in \mathcal{A}$ of finite measure  there exists a compact set $B\subset A $ with  $\mu(A \setminus  B) < \epsilon$ such that   $f \colon B \times X \to Y$  is continuous. It is important to emphasize that the above notion is a stronger than the Lusin integrand property, where it is only required that $t\tto \epi f_t$ is lsc.

\begin{remark}
	The Scorza-Dragoni  property has been used in several contexts.  From the best of our knowledge,  it appeared first in \cite{MR915076} in the context of continuous real functions. Then, similar results were obtained for set-valued maps (see, e.g., \cite{MR275484,MR430205,MR1491362})). Moreover, some authors use the terminology of Lusin' types theorems, due to the connection with the classical Lusin's Theorem. We refer to  \cite{MR616201,MR1121606,MR851255,MR915076} for further results and discussions.
\end{remark}
The next result provides some relations between Lusin and Scorza-Dragoni functions.
\begin{proposition}\label{binary_prop}
      Let $f \colon T\times X\to \Rex$ be a Lusin integrand and $g \colon T \times X\to \Rex$ a Scorza-Dragoni function. Let $F\colon\R\times \R\to \R$ be an usc function such that,   for all $x\in \R$, the map $u\mapsto F(u ,x )$ is nondecreasing. Then, the map  $(t,x)\mapsto F(f_t(x), g_t(x))$ is a Lusin integrand. Consequently, $f+g$ and $f-g$ are Lusin integrands. Moreover, if $g$ takes nonnegative values, then $f\cdot g$ is also a Lusin integrand.
\end{proposition}
\begin{proof}
    Let $\epsilon >0$ and $A\in \mathcal{A}$ of finite measure. Then, there exists a compact set $K\subset A$ with $\mu(A\setminus K)<\epsilon$ such that $K\ni t\tto \epi f_t$ is lsc and $g|_{K\times X}$ is continuous. We proceed to prove that  $K\ni t\tto \epi F(f_t(\cdot),g_t(\cdot))$ is lsc. Indeed, let $U\times I$ be an open set of $X\times \R$ and take $\hat{t}\in K$ such that there exists $(\hat{x},\hat{\alpha})\in U\times I$ and $F(f_{\hat{t}}(\hat{x}),g_{\hat{t}}(\hat{x}))<\hat{\alpha}$.  By the usc of $F$, there exists $J_1\times I_1$, neighborhood of $(f_{\hat{t}}(\hat{x}),g_{\hat{t}}(\hat{x}))$, such that $F(s,t)<\hat{\alpha}$ for all $(s,t)\in J_1\times I_1$. Moreover, since $g$ is continuous at $(\hat{t},\hat{x})$, there exists a neighborhood $V\times U_1$ of $(\hat{t},\hat{x})$ such that $g_t(x)\in J_1$ for all $(t,x)\in V\times U_1$, where $U_1\subset U$ and, since $K\ni t\tto \epi f_t$ is lsc,  $V_1 := \{ t\in T\colon \epi f_t\cap U_1\times I_1\neq\emptyset \}$ is a neighborhood of $\hat{t}$ in $K$. Now, if $t\in V\cap V_1$, then there exists $(x,\alpha)\in U_1\times I_1$ such that $f_t(x)\leq \alpha$ and $g_t(x)\in J_1$. Hence,  $F(f_t(x),g_t(x))\leq F(\alpha,g_t(x))<\hat{\alpha} $. Therefore, $\{ t\in T\colon\epi F(f_t(\cdot),g_t(\cdot))\cap (U\times I)\neq \emptyset \}$ is an open set of $K$, which implies that $K\ni t\tto \epi F(f_t(\cdot),g_t(\cdot))$ is lsc.
    \end{proof}
For uniform Lusin families, we have the following proposition.
\begin{proposition}\label{prop_un_lu}
The following assertions hold:
      \begin{enumerate}[label=(\alph*), ref=\ref{prop_un_lu}-(\alph*)]
          \item  \label{prop_lus_a} Let $\{f^i\colon T\times X\to\Rex\}_{i\in I}$ be a uniform Lusin family. Then, $f_t(x): =\inf_{i\in I} f^i_t(x)$ is a Lusin integrand. 
          \item    \label{prop_lus_b} Let $\{f^i\colon T\times X\to\Rex\}_{i\in I}$ be a uniform Lusin family. If $g\colon T\times X\to\Rex$ is a Scorza-Dragoni function, then $\{ f^i-g \}_{i\in I}$ is a uniform Lusin family.
          \item \label{prop_lus_c} Assume that $f^i \colon T\times X \to \Rex$ are Lusin integrand with respect to $\mu$, for all $i\in\N$. Then $(f^i)_{i\in\mathbb{N}}$ is a uniform Lusin family with respect to $\mu$.
          \item\label{prop_lus_d} The set-valued maps $(C_i\colon T\tto X)_{i\in I}$ are a uniform[$(d)$]  family of Lusin multifunctions if and only if $\{(t,x)\mapsto\delta_{C_i(t)}(x)\}_{i\in I}$ is a uniform Lusin family. 
      \end{enumerate}
\end{proposition}
\begin{proof}
    \textit{(a)} Let $A\in \mathcal{A}$ with $\mu(A)<\infty$ and $\epsilon >0$. As $\{f^i\}_{i\in I}$ is a uniform Lusin family, we have that there exists a compact set $K\subset A$ with $\mu(A\setminus K)<\epsilon$ such that $K\ni t\tto \epi f^i(t,\cdot)$ is lsc for all $i\in I$. We claim that $K\ni t\tto \epi f_t$ is lsc. Indeed, let $U\times J\subset X\times \R$ be an open set. We observe that if $\hat{t}\in K$ is such that there exists $(\hat{x},\hat{\alpha})\in U\times J$ with $f_{\hat{t}}(\hat{x})\leq\hat{\alpha}$, we can choose $\hat{\beta}\in J$ with $\hat{\beta}>\hat{\alpha}$. So,  $f_{\hat{t}}(\hat{x})<\hat{\beta}$ and there exists $i\in I$ such that $f^i_{\hat{t}}(\hat{x})<\hat{\beta}$. Moreover, since $K\ni t\tto \epi f^i(t,\cdot)$ is lsc, the set $V=\{ t\in K\colon \epi f^i_t\cap (U\times J)\neq\emptyset \}\in \mathcal{N}_{\hat{t}}$. Also, it is clear that $V\subset \{ t\in K\colon\epi f_t\cap (U\times J)\neq\emptyset \}$, which shows that $K\ni t\tto \epi f_t$ is lsc.
    
    \textit{(b)} Let $\epsilon >0$ and $A\in \mathcal{A}$ of finite measure. As $\{f^i\}_{i\in I}$ is a uniform Lusin family, there exists a compact set $K\subset A$ with $\mu(A\setminus K)<\epsilon$ such that $K\ni t\tto \epi f^i_t(\cdot)$ is lsc for all $i\in I$ and $g|_{K\times X}$ is continuous. We claim that $K\ni t\tto \epi(f^i_t(\cdot)-g_t)$ is lsc for all $i\in I$. Indeed, let $U\times J\subset X\times \R$ be an open set. We observe that if $\hat{t}\in K$ is such that there exists $(\hat{x},\hat{\alpha})\in U\times J$ with $f^i_{\hat{t}}(\hat{x})-g_{\hat t}(\hat x)\leq\hat{\alpha}$, then we can choose $\hat{\beta}\in J$ with $\hat{\beta}>\hat{\alpha}$. Hence,  $f^i_{\hat{t}}(\hat{x})-g_{\hat t}(\hat x)<\hat{\beta}$ and then there exists $\eta\in\R$ such that $f^i_{\hat{t}}(\hat{x})<\eta<g_{\hat t}(\hat x)+\hat{\beta}$. By  continuity, there exists a neighborhood $V\times U_1$ of $(\hat t,\hat x)$ with $U_1\subset U$ such that $g_t(x)>\eta-\hat \beta$ for all $(t,x)\in V\times U_1$. Moreover, as $K\ni t\tto \epi f^i_t(\cdot)$ is lsc, the set $V_1=\{ t\in K\colon\epi f^i_t\cap (V_1\times (-\infty,\eta))\neq\emptyset \}\in \mathcal{N}_{\hat{t}}$. Then, if $t\in V\cap V_1$, then there exists $(x,\alpha)\in V_1\times (-\infty,\eta)$ such that $f^i_t(x)\leq\alpha<\eta<g_t(x)+\hat\beta$. Hence, $f^i_t(x)-g_t(x)<\hat \beta$, which shows that $\{ t\in K\colon\epi (f^i_t-g_t)\cap (U\times J)\neq\emptyset \}$ is an open set for all $i\in I$.
    
    \textit{(c)} Let $\epsilon >0$ and $A\in\mathcal{A}$ of finite measure. For every $n\in\N$, let  $K_n\subset A$ be a compact set such that $\mu(K_n)< \epsilon\, 2^{-n}$ and $K_n\ni t\tto \epi f^n_t$. Then,  the set $K := \bigcap_{n\in\N} K_n$ satisfies $\mu(A\setminus K)<\epsilon$ and for all $n\in\N$, $K\ni t\tto \epi f^n_t$, which proves the result.
    
    \textit{(d)} If $(C_i\colon T\tto X)_{i\in I}$ is a uniform family of Lusin multifunctions,  then $\epi\delta_{C_i(t)} = C_i(t)\times [0,+\infty)$ for all $i\in I$. Let $\epsilon>0$ and $A\in\mathcal{A}$ of finite measure. Then,  there exists a compact set $B\subset A$ with $\mu(A\setminus B)<\epsilon$ and such that $K\ni t\tto C_i(t)$ is lsc for all $i\in I$. Let $i\in I$ and $U\times J$ be an open set of $X\times \R$, then  $\{t\in K\colon \epi \delta_{C_i(t)}\cap (U\times J)\neq\emptyset\}$ is open in $K$. Indeed, $\{t\in K\colon \epi \delta_{C_i(t)}\cap (U\times J)\neq\emptyset\}=\{t\in K\colon C_i(t)\cap U\neq\emptyset\}$ if $J\cap[0,+\infty)\neq \emptyset$ and  $\{t\in K\colon \epi \delta_{C_i(t)}\cap (U\times J)\neq\emptyset\}=\emptyset$ if $J\cap[0,+\infty)= \emptyset$. Therefore, $K\ni t\tto \epi\delta_{C_i(t)}$ is lsc for all $i\in I$. The reverse implication can be proved in a similar way.
\end{proof}

The following result provides a  sufficient condition so that the value function is a Lusin intergrand. 
\begin{proposition}
    Let $F\colon T\times X\tto Y$ and $f\colon T\times X\times Y\to\R$ be a Scorza-Dragoni mapping and a Scorza-Dragoni function, respectively. Then, the value function  
    $v_t(x):= \inf\{ f(t,x,y)\colon y\in F(t,x) \}$ is a Lusin integrand. \newline Moreover, if $F$ has locally compact values, i.e., for all $(t,x)\in T\times X$, there exists a neighborhood $U$ of $(t,x)$ such that $F(U)$ is relatively compact, then  $u_t(x):= \sup \{ f(t,x,y)\colon y\in F(t,x) \}$ is a Lusin integrand.
\end{proposition}
\begin{proof}
    Let us consider $\epsilon >0$ and $A\in \mathcal{A}$ of finite measure. By assumption,   there exists a compact set $B\subset A$ with $\mu(A\setminus B)<\epsilon$ and such that $f\colon B\times X\times Y\to \R$ is continuous and $F\colon B\times X\tto Y$ is lsc with  closed graph.\\
{\bf Claim~1:} \emph{The map $B\ni t\tto\epi v_t$ is lsc.} \newline
\noindent \emph{Proof of Claim 1}: Take $\hat{t}\in \{ t\in B\colon  \epi v_t\cap (U\cap I)\neq \emptyset \}$ where $U\times I$ is an open set of $X\times \R$.  Then, there exists $(\hat{x},\hat{\alpha})\in U\times I$ such that $v_{\hat{t}}(\hat{x})\leq \hat{\alpha}$. Hence, we can find $\hat{y}\in Y$ with $\hat{y}\in F(\hat{t},\hat{x})$ such that $f(\hat{t},\hat{x},\hat{y})\leq\hat{\alpha}$. We can choose $\beta>\hat{\alpha}$, because $I$ is open and by continuity there exists $O_1\times U_1\times V_1$, neighborhood of $(\hat{t},\hat{x},\hat{y})$, such that $f(t,x,y)<\beta$ for all $(t,x,y)\in O_1\times U_1\times V_1$ where $U_1\subset U$. On the other hand, since $F|_B$ is lsc, there exists $O_2\times U_2$ neighborhood of $(\hat{t},\hat{x})$ such that $O_2\times U_2\subset\{ (t,x)\in B\times X\colon  F(t,x)\cap V_1\neq \emptyset \}$, where $U_2\subset U$. Then, if $t\in O_1\cap O_2$, we can take $x\in U_1\cap U_2$ and find $y\in V_1$ such that $y\in F(t,x)\cap V_1$. Hence, $f(t,x,y)<\beta\Rightarrow v_t(x)<\beta$.  Therefore $O_1\cap O_2\subset \{ t\in B\colon \epi v_t\cap (U\cap I)\neq \emptyset \}$, which proves the claim. \\
{\bf Claim~2:} \emph{If $F$ has locally compact values, then $ B\ni t\tto \epi u_t $ is lsc.} \newline
\noindent \emph{Proof of Claim 2}: Indeed, let us consider an open set $U\times I$ of $X\times \R$ and assume that $D:=\{ t\in B\colon \epi u_t\cap U\times I\neq\emptyset \}$ is not open. Then, there exists $\hat{t}\in D$ such that for all neighborhood $V$ (in $B$) of $\hat{t}$, there exists $t\in V$ such that $\epi u_t\cap (U\times I)=\emptyset$. In particular, we can take $V_n = \mathbb{B}_{1/n}(\hat t)\cap B$. Hence,  for all $n\in \N$, there exists $t_n\in \mathbb{B}_{1/n}(\hat t)\cap B$ such that $\epi u_{t_n}\cap (U\times I)=\emptyset$ and it is clear that $t_n\to \hat t$. On the other hand, there exists $V_1\times U_1\in \mathcal{N}_{(\hat{t},\hat{x})}$ such that $F(V_1\times U_1)$ is relatively compact. Moreover, there exists $(\hat{x},\hat{\alpha})\in U\times I$ such that $u_{\hat{t}}(\hat{x})\leq\hat{\alpha}$ and taking $\gamma\in I$ with $\gamma>\hat{\alpha}$. we obtain that for all $n\in\N$, $u(t_n,\hat{x})>\gamma$. Then,  there exists $y_n\in F(t_n,\hat{x})$ such that $f(t_n,\hat{x},y_n)>\gamma$. Since $t_n\to \hat{t}$, there exists $N\in\N$ such that $t_n\in V_1$ if $n\geq N$, then $y_n\in F(V_1\times U_1)$ for $n\geq N$. Hence, there exists $\hat{y}\in Y$ and a subsequence $\{ y_{n_k} \}$ such that $y_{n_k}\to \hat{y}$. Since $F|_B$ has closed graph,  $\hat{y}\in F(\hat{t},\hat{x})$. However,   $f(t_{n_k},\hat{x},y_{n_k})>\gamma$ for all $k\in\N$ and by continuity $f(\hat{t},\hat{x},\hat{y})\geq \gamma>\hat{\alpha}$. Thus,  $u_{\hat{t}}(\hat{x})>\hat{\alpha}$, which is a contradiction. Therefore, $D$ is open and  $B\ni t\tto \epi u_t$ is lsc.
\end{proof}

\section{Measurable selections for admissible  integrands and multifunctions}\label{sect-5}

This section introduces a class of integrands and multifunctions which are amenable for selection theorems. The results of this section are crucial in developing calculus rules for conjugate and subdifferential of non necessarily convex integral functionals. 

We say that $f\colon T\times X\to\Rex$ is an \emph{admissible integrand} if there exists a uniform Lusin family  $(f^i\colon T\times X\to\Rex)_{i\in I}$ such that  $f_t^i\in \Gamma(X)$ for all $i\in I$ and $t\in T$ and $\epi f_t = \cl(\epi (\inf_{i\in I} f^i_t))$ for all $t\in T$.
\begin{remark}\label{lemma_inf_}
It is important to emphasize that that for two proper functions $f,g$ the condition  $\epi f = \cl(\epi g)$ is equivalent to say that $f$ coincides with the lsc envelope of $g$. Hence, if $\epi f = \cl(\epi g)$, then $\inf_{x\in X}f(x)=\inf_{x\in X}g(x)$. 
\end{remark}

\begin{remark}\label{adm_normal}
The notion of admissible integrand is related to the concept of normal integral. Recall that $f\colon T\times X\to \Rex$ is normal integrand if $f$ is $\mathcal{A}\otimes\mathcal{B}(X)$ measurable and the maps $x\mapsto f_t(x)$ are lsc for all $t\in T$ (see, e.g., \cite{MR1485775,MR1491362,MR0467310}). Hence, when $X$ is separable,  both concepts of integrand coincide. Indeed,  it is clear that that an admissible integrand is a normal integrand.  Reciprocally, by virtue of  the  Castaing representation (\cite{MR0467310}), we can find a  sequence of  measurable functions $x_n\colon D\to X$ and $\alpha_n\colon D\to\R$ such that $\epi f_t = \cl(\{(x_n(t),\alpha_n(t))\}_{n\in\N})$ for all $t\in D$, where $D = \dom (\epi f)$.  Hence, for each $n\in\N$, we can consider the function $f^n_t (x):=\alpha_n(t)+\delta_{\{x_n(t)\}}(x)$ for $t\in D$ and $f^n_t(x) = +\infty$ for $t\notin D$. It follows that  $f^n_t\in\Gamma(X)$ for all $n\in\N$.  Moreover,  it is clear that $f^n$ is a Lusin integrand. Finally,  by Proposition \ref{prop_lus_c}, $(f^n)_{n\in\N}$ is a uniform Lusin family and $\epi f_t = \cl{\epi (\inf_{n\in\N}f^n_t)}$, which proves  that  $f$ is an admissible integrand.
\end{remark}

The notion of admissible integrand can be extended to multifunctions.  We say that $C\colon T\tto X$ is an \emph{admissible multifunction} if there exists a uniform family of Lusin multifunctions $(C_i\colon T\tto X)_{i\in I}$ with closed and convex values such that $C(t) = \cl\left(\bigcup_{i\in I} C_i(t)\right)$ for all $t\in T$. 
\begin{proposition}\label{mult_adms} A multifunction $C\colon T\tto X$ is admissible if and only if $(t,x)\mapsto \delta_{C(t)}(x)$ is an admissible integrand.
\end{proposition}
\begin{proof}
    If $C\colon T\tto X$ is admissible, then $C(t) = \operatorname{cl} \left( \bigcup_{i\in I} C_i(t) \right)$ for all $t\in T$, for some uniform family of Lusin multifunctions $(C_i\colon T\tto X)_{i\in I}$. By virtue of Proposition \ref{prop_lus_d},   $((t,x)\mapsto\delta_{C_i(t)}(x))_{i\in I}$ is a uniform Lusin family.\newline  Moreover, $\epi \delta_{C(t)} = \cl{\epi(\inf_{i\in I}\delta_{C_i(t)})}$, which shows that $(t,x)\mapsto \delta_{C(t)}(x)$ is an admissible integrand.  \newline 
    Reciprocally, assume that $(t,x)\mapsto \delta_{C(t)}(x)$ is an admissible integrand.  Then, there exists a uniform Lusin family $(f^i)_{i\in I}$ such that 
    $$
    \epi\delta_{C(t)} = \cl{\epi(\inf_{i\in I}f_t^i(t,\cdot))} = C(t)\times [0,+\infty)  \text{ and }   f^i_t \in\Gamma(X) \text{ for all  }i\in I. 
    $$
    For every $j\in J :=I\times (0,1]$ we consider the multifunction ${C}_j\colon T\tto X$  defined by  ${C}_j(t) = \cl(\{x\in X\colon  f^i_t(x)< \epsilon\})$, where $j=(i,\varepsilon)$. It is clear that $C_j$ takes closed and convex values, and  $\{C_j\colon T \tto X \}_{ j \in J}$ is a uniform Lusin family of multifunctions. \smartqed
\end{proof}

The next result provides  a measurable selection theorem for sublevel sets of an admisible integrand. It can be seen as a measurable counterpart of  Lemma \ref{lemma_selection} without convexity and lower semicontinuity of  the epigraph multifunction.
\begin{theorem}\label{theorem_selection_2} 
Let $f\colon T \times X\to \Rex $ be an admissible integrand and consider  $A\in\mathcal{A}$ such that there exists a measurable function $\alpha\colon A\to\Rex$ such that $m_f(t)<\alpha(t)$ for all $t\in A$. Then, there exists a strongly measurable function $x\colon A\to X$ such that $f_t(x(t))\leq \alpha(t)$ $\mu$-a.e. in $A$.    
\end{theorem}

\begin{proof}
Due to Remark \ref{lemma_inf_}, it is enough to assume that $f_t$ is the pointwise infimum of a uniform Lusin family $(f^i_t)_{i\in I} \subset \Gamma(X)$. Hence, due to Proposition \ref{prop_lus_a},  $f$ is a Lusin integrand. Set $A':= \{x\in A\colon  \alpha(x)<+\infty\}$ and choose a sequence of compact sets $(B_n)_{n}$ of finite measure with  $\mu(A'\setminus (\bigcup_{n\in\N}B_n))=0$ and such that for all $n\in\N$, $\alpha \colon B_n\to\R$ is continuous and $B_n\ni t\tto \epi f^i_t$ is lsc, for all $i\in I$. Consider  $n\in\N$ and  $t\in B_n$.  Since  $m_f(t)<\alpha(t)$, there exists $i(t)\in I$ such that $m_{f^{i(t)}}(t)<\alpha(t)$. Hence,  by Lemma \ref{f_l_b}, $m_{f^{i(t)}}$ is usc on $B_n$ and since $\alpha|_{B_n}$ is continuous, there exists a neighbourhood $U_t$ of $t$ (relative to the topology on $B_n$)  such that $m_{f^{i(t)}}(s)<\alpha(s)$ for all $s\in U_t$. We observe that $(U_t)_{t\in B_n}$ is an open cover of $B_n$. Hence, by compactness of $B_n$, there exist $t_1,...,t_j\in B_n$ such that $B_n = \bigcup_{k=1}^j U_{t_k}$. Now,  by Lemma \ref{lemma_selection}, there exists continuous functions $x_{t_k}\colon U_{t_k}\to X$ such that $f^{i(t_k)}_s(x_{t_k}(s))\leq \alpha(s)$ for all $s\in U_{t_k}$. Hence, $f_s(x_{t_k}(s))\leq \alpha(s)$ for all $s\in U_{t_k}$. We can consider the function $x_n\colon B_n\to X$ given by $x_n(t) = \sum_{k=1}^j x_{t_k}(t)\1_{A_k}$ where $A_k = U_{t_k}\setminus \bigcup_{\ell=1}^{k-1}U_{t_{\ell}}$ for $k>1$ and $A_1 = U_{t_1}$.  We have that  $f_t(x_n(t))\leq \alpha(t)$ for all $t\in B_n$. We define the strongly measurable function $x'\colon A'\to X$ as $x'(t) = \sum_{n=1}^\infty x_n(t)\1_{B_n}$.  We have that $f_t(x(t))\leq \alpha(t)$ for all $t\in \bigcup_{n\in\N}B_n$, and for $t\in A\setminus A'$, we define $y(t) = x'\in X$ (where $x'$ is any element in $X$). Finally, by defining $x(t):= x'(t)\1_{A'} + y(t)\1_{A\setminus A'}$, we obtain the result. \smartqed
\end{proof}

The next result provides a measurable selection theorem for admissible multifunctions.
    \begin{theorem}\label{selection_thm}
Let $C\colon T\tto X$ be an admissible multifunction and  $A\in\mathcal{A}$. If $C$ takes nonempty values on $A$, then there exists a strongly measurable function $x\colon A\to X$ such that $x(t)\in C(t)$ $\mu$-a.e. $t\in A$.
    \end{theorem}
    \begin{proof}
        By virtue of Proposition \ref{mult_adms}, $(t,x)\mapsto\delta_{C(t)}(x)$ is an admissible multifunction. Taking into account that  $C(t)\neq \emptyset$ for all $t\in A$, it is clear that $\inf_{x\in X}\delta_{C(t)}(x)=0<1$, for all $t\in A$. Set  $\alpha(t):= 1$ for all $t\in A$. We have that $\inf_{x\in X}\delta_{C(t)}(x)<\alpha(t)$ for all $t\in A$. Hence, by Theorem \ref{theorem_selection_2},  there exists a strongly measurable function $x\colon A\to X$ such that $\delta_{C(t)}(x(t))\leq \alpha(t) = 1$ $\mu$-a.e., which implies that $x(t)\in C(t)$ $\mu$-a.e.. \smartqed
    \end{proof}
    
    \begin{remark} It is usual that for existence theorems for 
 differential inclusions be assumed the existence of a measurable selection of the multifunction defining the differential inclusion (see, e.g.,  \cite[Theorem 10.5]{MR1189795}). For nonseparable spaces this is not an easy task. Hence, Theorem \ref{selection_thm} is a reasonable alternative for finding such measurable selections in general spaces.
    \end{remark}

\section{Conjugate of integral functions  and consequences}\label{sect-6}

In this section, we compute the conjugate of non-necessarily convex integral functionals and derive calculus rules for the convex subdifferential.

We start by showing the measurability of a Lusin integrand and its conjugate composed with a strongly measurable function.

\begin{lemma}\label{lemma_lusin}
	Let $\mu$ be a $\sigma$-finite  inner regular measure over a metric space $T$ and $f\colon T \times X \to \Rex$ a Lusin integrand such that $f_t$ is lsc for all $t\in T$.  Then, the maps $t \mapsto f_t(x(t))$ and $t\mapsto f^\ast_t(x^\ast(t))$ are $\mathcal{A}$-measurable for any strongly measurable functions $x\colon T\to X$ and $x^\ast\colon  T\to  X^\ast$.
\end{lemma}

\begin{proof}
	Since $(T,\mathcal{A},\mu)$ is $\sigma$-finite, it is enough to prove the measurability over sets of finite measure.  Let $A\in \mathcal{A}$ of finite measure. Consider an increasing sequence of compact sets $B^1_k \subset A$ such $\mu(A\setminus  \bigcup_{k \in \mathbb{N}} B^1_k )=0$ and $B^1_k \colon t \tto \epi f_t$ is lsc. By virtue of Lemma \ref{first_lemma}, the map $(t,x^\ast) \mapsto f_t^\ast(x^\ast)$ is lsc over $B^1_k\times X^\ast$. Hence, the map $B_k^1\ni t\mapsto f^\ast_t(x^\ast(t))$ is measurable for all $k\in\mathbb{N}$ and, therefore, $t\mapsto f^\ast_t(x^\ast(t))$ is $\mathcal{A}$-measurable  over $A$ (recall that  $(T,\mathcal{A},\mu)$ is complete). Now, by applying Proposition \ref{lusin} to $x|_{_A} $, there exist an increasing sequence of  compact sets $B_k^2 \subset A$ such that $\mu(A\setminus  \bigcup_{k \in \mathbb{N}} B^2_k )=0$ and $B^2_k \colon t \mapsto x(t)$ is continuous. Set $B_k := B_k^1 \cap B_k^2$.  Then, by Lemma \ref{f_l_c}, $t \mapsto  f_t(x(t))$  is measurable over $B_k$.  Thus, if $B:= \bigcup_{k \in \mathbb{N}} B_k$, then we have that $\mu(A\setminus  B)=0$, and the  map $  B\ni  t \mapsto  f_t(x(t))$ is $\mathcal{A}$-measurable, which implies the measurability of $t\mapsto f_t(x(t))$. 
\end{proof}
In what follows, we consider the following integral functions:
\begin{equation*}
	\Intf{f}(x)=\int_T    f_t(x(t))  d \mu \text{ and }\Intf{f^\ast}(x^\ast)=\int_T    f^\ast_t(x^\ast(t))  d \mu,
\end{equation*}
where $x\colon T\to X$ and $x^\ast \colon T\to X^\ast$ are strongly measurables. According to Lemma \ref{lemma_lusin}, if $f$ is a Lusin integrand, then $\Intf{f}$ and $\Intf{f^\ast}$ are well-defined.  However,  these functionals cannot be considered for weakly$^\ast$ measurable functions, as it is shown in the following example.
\begin{example}[Non measurability of $t\mapsto f^\ast_t(x^\ast(t))$]\newline
Let us consider $X = \ell^2([0,1])$ and $T=[0,1]$ with the Lebesgue measure. Let $(e_t)_{t\in [0,1]}$ the canonical basis of $X$ and   $I\subset [0,1]$ be a non Lebesgue measurable set.  Define $x^\ast \colon T\to X^\ast$  and  $f\colon T\times X \to \Rex$ as $x^{\ast}(t):=e_t \1_{I}(t)-e_t \1_{I^c}(t) $ and $f_t(x):=\delta_{C}(x)$,
where  $C:= \{ x\in X\colon  0\leq x_t\leq 1 \textrm{ for all } t\in [0,1] \}$. Then, $f^{\ast}_t(y^{\ast})=\sup_{x\in C}\langle y^{\ast},x\rangle$. Moreover, if $x\in X$, then the set $\{t\in [0,1]\colon x_t\neq 0\}$ is countable. Hence, the map $t\mapsto \langle x^{\ast}(t),x\rangle$ is measurable for all $x\in X$, which implies that $x^{\ast}$ is weakly$^{\ast}$ measurable. However,  $\{ t\in T\colon  f^\ast_t(x^\ast(t))>0 \} =I$, which  implies that $t\mapsto f^\ast_t(x^\ast(t))$ is not measurable. \qed
\end{example}

The following result establishes that any admissible integrand is a Lusin integrand. 
\begin{lemma}\label{lemma_uniform}
      Let $f$ be an admissible integrand. Then, $f$ is a Lusin integrand. Consequently, the maps $t\mapsto f_t(x(t))$ and $t\mapsto f^\ast_t(x^\ast(t))$ are $\mathcal{A}$-measurables for any strongly measurable functions $x\colon T\to X$ and $x^\ast\colon T\to X^\ast$.
\end{lemma}

\begin{proof}
    Let $(f^i)_{i\in I}$ be a uniform Lusin family such that $$\epi f_t = \cl{\epi (\inf_{i\in I} f^i(t,\cdot ))} \text{ with } f_t^i \in \Gamma(X) \text{ for all } (i,t)\in I\times T.  $$ Set $g_t(x) := \inf_{i\in I} f^i_t(x)$.  It follows that  $f = \cl \ g$ and by Lemma \ref{prop_lus_a}, $g$ is a Lusin integrand. Hence, by Lemma \ref{lus_c}, $f$ is a Lusin integrand. The second  assertion follows from  Lemma \ref{lemma_lusin}.
\end{proof}
Let $\mathcal{L}$ be the  space of strongly measurable functions $x\colon T\to X$. We say that $\mathcal{L}$ is \emph{decomposable} (see, e.g., \cite{MR0467310,MR1491362}) if for every $x\in \mathcal{L}$,  and every strongly measurable function $y\colon T\to X$  with compact range the function $x\1_{A^c}  + y \1_A\in \mathcal{L}$ for all $A \in \mathcal{A}$ of finite measure. 

Now, we state the main result of this section, which corresponds to the computation of the conjugate of an integral function in a decomposable space of functions with values on a non-necessarily separable Banach space. It is worth to emphasize that the above result does not require the convexity of the integrand.
\begin{theorem}[Conjugate of an integral functional I]\label{Theorem01} \newline
Let $f$ be an admissible integrand and $\mathcal{L}$ be a decomposable space of measurable functions such that there exists some $u \in \mathcal{L}$ with $\Intf{f}(u) <+\infty$.  Then, for any   strongly measurable function $x^\ast\colon T\to X^\ast $   
	\begin{align}\label{conjugate}
		\Intf{f^\ast}(x^\ast)=\sup_{x\in  \mathcal{L}} \int_T  \left(  \langle x^\ast(t), x(t)\rangle - f_t(x(t))  \right) \dmu.
	\end{align}
\end{theorem}
\begin{proof} If  $\Intf{f}(u)=-\infty$, then the result holds trivially. Otherwise, if $\Intf{f}(u)\in \mathbb{R}$, by definition of  $f^\ast_t(x^\ast(t))$, we can consider a sequence of integrable functions $(\ell_j)_{j\in \N }$   such that $\ell_j(t) \leq \ell_{j+1} (t)< f_t^\ast(x^\ast(t))$ $\mu$-a.e. and  
	\begin{equation*}
		\sup_{ j\in \mathbb{N}} \int_T \ell_j (t) \dmu=   \Intf{f^\ast}(x^\ast).
	\end{equation*}
	\noindent The rest of proof is divided into two claims:

	\noindent {\bf Claim~1:} \emph{The function $\psi_t(x) :=f_t(x) -  \langle x^\ast(t),x\rangle $ is an admissible integrand.} \newline
\noindent \emph{Proof of Claim 1}: Let us consider a uniform Lusin family $(f_t^i)_{i\in I}$  such that $\epi f_t = \cl\left( \epi(\inf_{i\in I} f_t^i )\right)$.   It is clear that 
$\epi \psi_t = \cl\left( \epi(\inf_{i\in I} g^i_t)\right)$, where $g_i(t,x) := f_i(t,x) - \langle x^\ast(t),x\rangle$. Moreover, by  Proposition \ref{lusin}, for all $\epsilon>0$ and $A\in \mathcal{A}$ of finite measure, there exists a compact set $K\subset A$ with $\mu(A\setminus K)<\epsilon$ such that $x^\ast\colon K\to X$ is continuous. Hence, $K\times X\ni (t,x) \mapsto \langle x^\ast(t),x\rangle$ is continuous. Consequently, by Lemma \ref{prop_lus_b},   $(g_i)_{i\in I}$ is a uniform Lusin family, which proves the claim.

\noindent {\bf Claim~2:} \emph{Theorem \ref{Theorem01} holds.} \newline
\noindent \emph{Proof of Claim 2}: Let us consider $\epsilon >0$ and $j\in\mathbb{N}$. By Lemma \ref{lemma_uniform},  $\psi$ is a Lusin integrand. Hence, we can apply Theorem \ref{theorem_selection_2} to get  the existence of a strongly measurable function $y\colon T\to X$ such that $\psi_t(y(t))\leq -\ell_j (t),$ for all $t\in T$. Now,  there is  a compact set $B$ of finite measure such that  $\int_{ B^c } |\ell_j(t)| \dmu < \epsilon$ and $\int_{  B^c } |\psi_t(u(t))| \dmu < \epsilon$. Hence, by virtue of Proposition \ref{lusin},  we can assume that $B\ni t\mapsto y(t)$ is continuous, which implies that $y(B)$ is compact.  Moreover,
	\begin{align*}
		\int_B \ell_j(t) \dmu \leq  \int_B (   \langle x^\ast(t),y(t)\rangle - f_t(y(t)) ) \dmu.
	\end{align*}
Hence, by extending $y$ as zero out of $B$ and using  the  decomposability of $\mathcal{L}$, we obtain that 
	$z(t):= u(t) \1_{B^c} + y(t) \1_{B} $ belongs to $\mathcal{L}$. Furthermore, 
	\begin{align*}
		\int_T \ell_j(t) \dmu &=  \int_{B}\ell_j(t)\dmu + \int_{B^c}\ell_j(t)\dmu\leq  \int_B \left(    \langle x^\ast(t),y(t)\rangle - f_t(y(t)) \right) \dmu+\epsilon \\
		&\leq \epsilon + \int_{T} \left(    \langle x^\ast(t),z(t)\rangle - f_t(z(t)) \right) \dmu+\int_{B^c}\psi(t,u(t)) \dmu\\
		&\leq 2\epsilon + \int_{T}\left(    \langle x^\ast(t),z(t)\rangle - f_t(z(t)) \right) \dmu\\
		&\leq 2\epsilon + \sup_{x\in  \mathcal{L}}  \int_T  \left(  \langle x^\ast(t), x(t)\rangle - f_t( x(t))  				 \right)d \mu,
	\end{align*}
	which implies, since $\epsilon >0$ is arbitrary and $j \in \mathbb{N}$, the inequality $\leq $ in \eqref{conjugate}. The opposite inequality always holds.  The proof is then finished.
\end{proof}

\noindent The next corollary gives a characterization of optimal solutions for optimization problems  over decomposable spaces.
\begin{corollary}  Under the assumptions of Theorem \ref{Theorem01}, an element $\bar{x} \in \mathcal{L}$ solves the problem  $ \min_{ x\in \mathcal{L}}  \Intf{f}(x)$ if and only if  $\bar{x}(t)\in \operatorname{argmin}_{x\in X}f_t(x)$ for almost all $t\in T$.
\end{corollary}

\begin{proof}
By Theorem \ref{Theorem01},  $\Intf{f}(\bar{x}) = \inf_{x\in \mathcal{L}} \Intf{f}(x) 
    =-\mathcal{I}_{f^\ast}(0)
    = \int_T \inf_{ x\in X} f_t(x) \dmu$. Hence,  $f_t(\bar{x}(t))= \inf_{ x\in X} f_t(x)$ for a.e. $t\in T$, which ends the proof.
\end{proof}

 For every $p\in [1,+\infty]$, the space $L^p(T,X)$ is  decomposable. In addition, when $X$ is an Asplund space and $p<+\infty$, its topological dual  is given by  $L^q(T,X^\ast)$, where $1/p+1/q=1$ (see, e.g.,  \cite{MR0453964}). On the other hand, the dual space of $L^\infty(T,X)$ can be identify  with  the direct sum $L^1(T,X^\ast_{w^\ast})\oplus L^{\text{sing}}(T,X)$, where $L^{\text{sing}}(T,X)$ is the space of singular measures defined on $L^\infty(T,X)$ (see, e.g., \cite[Theorem 4.1]{levin1975convex}).
\begin{corollary}\label{cor1}
Let $X$ be  an Asplund space and assume that $p \in [1,+\infty)$. Let  $f\colon T\times X \to \Rex$ an admissible integrand such that $\Intf{f}(u)<+\infty$ for some $u \in L^p(T,X)$. Then, for all $\epsilon\geq 0$, the following formulas hold:
\begin{align*}
   (\Intf{f})^\ast(x^\ast)&=	\Intf{f^\ast}(x^\ast) \textrm{ for every } x^\ast \in L^q(T, X^\ast), \textrm{ with } q = \frac{p}{p-1},\\
    \partial_\epsilon \Intf{f}(x) &=\bigcup_{\eta\in\mathcal{Z}(\epsilon)}\left\{  x^\ast \in L^q(T, X^\ast) \colon x^\ast(t) \in \partial_{\eta(t)} f_t(x(t)) \quad \mu\text{-a.e.}\right\}, 
\end{align*}
where $\mathcal{Z}(\epsilon):= \{ \ell\in L^1(T)\colon\ell\geq 0  \ \mu\text{-a.e. and }\int_T\ell(t)\dmu\leq \epsilon \}.$
\end{corollary}
\begin{proof} The first formula is a direct consequence of Theorem \ref{Theorem01}. To prove the second one, let  $x^\ast \in L^q(T,X)$. Then, by Theorem \ref{Theorem01}, we have that 
\begin{equation*}
x^\ast\in \partial_\epsilon \Intf{f}(x) \iff  \int_T \eta(t) \dmu\leq \epsilon,
\end{equation*}
where $\eta(t) := f_t(x(t)) - \langle x^\ast(t),x(t)\rangle + f^\ast_t(x^\ast(t))\geq 0$ for a.e.  $t\in T$ by virtue of the Young-Fenchel inequality. Hence, $\eta\in\mathcal{Z}(\epsilon)$. Moreover, $\eta(t)\leq 0$ if and only if $x^\ast(t)\in \partial_{\eta(t)}f_t(x(t))$. Hence, we have proved the inclusion $\subset$ in the second formula.  The reverse inclusion follows from \cite[Theorem 2.4.2]{MR1921556}.
\end{proof}

The following result gives a formula for the conjugate of $\Intf{f}$ over $L^\infty(T,X)$ for non-necessarily separable spaces. Hence, we extend  \cite[Theorem 6.4]{levin1975convex} (see also \cite[Theorem 1]{MR4261271}),  where $X$ was assumed to be separable Banach space.  In the following theorem, we consider $\Intf{f}$ defined over $L^\infty(T,X)$.
\begin{theorem}[Conjugate of an integral functional II ]
 Let $X$ be an Asplund space and $f\colon T\times X \to \Rex$ be an  admissible integrand. Assume that  $\Intf{f}(u)<+\infty$ for some $u \in L^\infty(T,X)$. Then, for every $x^\ast \in L^1(T,X^\ast)$ and $\lambda^\ast\in L^\text{sing}(T,X)$, the following formula holds:
\begin{align}\label{form_conj_L_inf}
    (\Intf{f})^\ast(x^\ast +\lambda^\ast )&=	\Intf{f^\ast}(x^\ast) + \sigma_{ \dom 	\Intf{f} } (\lambda^\ast).
\end{align}
Moreover, for all $\epsilon\geq 0$ and $x\in \dom\Intf{f}\cap L^\infty(T,X)$, if $v^\ast\in L^1(T,X^\ast)$ and $\lambda^\ast\in L^{\text{sing}}(T,X)$, then $v^\ast +\lambda^\ast  \in \partial_\epsilon\Intf{f}(x)$ if and only if there exist $\epsilon_1,\epsilon_2\geq 0$ and $\ell\in \mathcal{Z}(\epsilon_1)$ with $\epsilon_1+\epsilon_2\leq \epsilon$ such that $$v^\ast(t)\in\partial_{\ell(t)}f_t(x(t)) \ \mu\text{-a.e. and }\lambda^\ast\in N_{\dom \Intf{f}}^{\epsilon_2}(x). $$
\end{theorem}
\begin{proof}
It is clear that $(\Intf{f})^\ast(x^\ast +\lambda^\ast )\leq \Intf{f^\ast}(x^\ast) + \sigma_{\dom \Intf{f}}(\lambda^\ast)$. To prove the opposite inequality, let us consider $\epsilon >0$, and $r_1,r_2 \in \R$ such that $r_1 <\Intf{f^\ast}(x^\ast ) $ and $r_2 < \sigma_{\dom \Intf{f}}(\lambda^\ast)$. Hence,  we can find $x_1,x_2\in \dom \Intf{f} $ such that $\langle x^\ast,x_1\rangle-\Intf{f}(x_1)\geq r_1$  and $\langle\lambda^\ast,x_2\rangle\geq  r_2$. 
     Since $x_1,x_2\in \dom \Intf{f}$, there exists $A\in \mathcal{A}$ of finite measure such that for $i=1,2$
     \begin{equation*}
         \int_{A^c} |\langle x^\ast(t),x_i(t)\rangle-f_t(x_i(t))|\dmu <\epsilon.
     \end{equation*}
    Since $\lambda^\ast$ is singular, there exists and increasing sequence $(A_n)\subset \mathcal{A}$ such that $\mu(T \backslash \bigcup_{n\in\N} A_n)=0$ and for all $B\in \mathcal{A}$ of finite measure, $\lambda^\ast|_{A_n\cap B} = 0$ for all $n\in \N$. Next, define $y_n(t)=x_1(t)\1_{A\cap A_n}(t)+x_2(t)\1_{(A\cap A_n)^c}(t)$. Then, $\langle\lambda^\ast,y_n\rangle = \langle\lambda^\ast,x_2\rangle$ for all $n\in\N$. Moreover, since  $\mu(A\cap A_n^c)$ converges to $0$, we can find  $k\in \mathbb{N}$ such that for $i=1,2$ $$ \int_{A\cap A_k^c}|\langle x^\ast(t),x_i(t)\rangle - f_t(x_i(t))| \dmu<\epsilon. $$
    Then, by definition of conjugate,
    \begin{eqnarray}
        \nonumber (\Intf{f})^\ast(x^\ast+\lambda^\ast) \geq  \langle x^\ast+\lambda^\ast,y_k\rangle -\Intf{f}(y_k)
        \nonumber =\langle \lambda^\ast,x_2\rangle+\langle x^\ast,y_k\rangle -\Intf{f}(y_k).
    \end{eqnarray}
Moreover, if $h_i(t):=\langle x^\ast(t),x_i(t)\rangle - f_t(x_i(t)) \dmu$ for $i=1,2$, we obtain that
\begin{equation*}
    \begin{aligned}
    \langle x^\ast,y_k\rangle -\Intf{f}(y_k)&=  \int_{A\cap A_k} h_1(t)\dmu+\int_{(A\cap A_k)^c} h_2(t)\dmu \\
&= \langle x^\ast,x_1\rangle-\Intf{f}(x_1) - \int_{A^c} h_1(t)\dmu- \int_{A\cap A_k^c} h_1(t)\dmu\\
&+ \int_{A^c} h_2(t)\dmu +\int_{A\cap A_k^c} h_2(t)\dmu\\
&\geq \langle x^\ast,x_1\rangle-\Intf{f}(x_1)-4\epsilon.
    \end{aligned}
\end{equation*}
    Therefore, $   (\Intf{f})^\ast(x^\ast+\lambda^\ast)\geq \langle \lambda^\ast,x_2\rangle+\langle x^\ast,x_1\rangle-\Intf{f}(x_1)-4\epsilon  \geq r_1 +r_2 - 4\epsilon$. Since, $r_1,r_2$ and $\epsilon$ are arbitrary, we obtain formula \eqref{form_conj_L_inf}.  \newline 
    To prove the last assertion, let us consider $\epsilon \geq 0$ and $x\in\dom\Intf{f}$. Then, by definition and formula \eqref{form_conj_L_inf}, we have that
    \begin{eqnarray*}
        v^\ast+\lambda^\ast \in\partial_\epsilon\Intf{f}(x) &\iff& \Intf{f}(x) + (\Intf{f})^\ast(v^\ast+\lambda^\ast)-\langle v^\ast+\lambda^\ast,x\rangle\leq\epsilon\\
        &\iff & \Intf{f}(x) + \Intf{f^\ast}(x^\ast)-\langle v^\ast,x\rangle + \sigma_{\dom\Intf{f}}(\lambda^\ast)-\langle\lambda^\ast,x\rangle\leq \epsilon.
    \end{eqnarray*}
    It is enough to  set $\epsilon_1 := \Intf{f}(v) + \Intf{f^\ast}(v^\ast)-\langle v^\ast,v\rangle$,  $\epsilon_2 := \sigma_{\dom\Intf{f}}(\lambda^\ast)-\langle\lambda^\ast,x\rangle$ and  $\ell(t):=f_t(v(t)) +f^\ast_t(v^\ast(t))-\langle v^\ast(t),v(t)\rangle$ to conclude the proof.
\end{proof}

 As an application of the above theorem, we can compute the subdifferential and conjugate of   the  \emph{expected functionals} $E_f\colon X \to \Rex$ defined by $\Expf{f}(x):=\int_T f_t(x) \dmu$, where $f$ is a Lusin integrand with  $f_t\in \Gamma_0(X)$ for all $t\in T$. These functionals appear commonly in stochastic programming, where they are used to study  the expected formulation of mathematical programming problems with uncertainty (see, e.g.,  \cite{MR4261271,MR3947674,MR4062793,MR4295325,MR4350897}). Before presenting these results, we need a sufficient condition for the continuity of $	\Intf{f}$ over $L^{\infty}(T,X)$.
\begin{lemma}\label{lemma1} Assume that $X$ is an Asplund space and consider a Lusin integrand $f\colon T\times X \to \Rex$ with $f_t\in \Gamma_0(X)$ for all $t\in T$. Let $x_0 \in L^\infty(T,X)$ such that $\Intf{f}(x_0) \in \mathbb{R}$ and there exist $k\in L^1(T)$  and $\eta >0$ such that 
      \begin{equation*}
          f_t(x ) \leq k(t), \textrm{ for all } x\in \mathbb{B}_{\eta}(x_0(t)) \textrm{ and all } t\in T.
      \end{equation*}
      Then, $\Intf{f}$ is continuous at $x_0$ with respect to the topology of $L^\infty(T,X)$.
\end{lemma}

\begin{proof} Under the assumptions of the lemma, $\Intf{f}$ is a proper convex   function, which is bounded from above on a neighbourhood   of $x_0$.  Then, then the result follows from \cite[Corollary 2.2.13]{MR1921556}.
\end{proof}

\begin{theorem}\label{subdif_ef}
 Assume that  $X$ is an Asplund space and consider a Lusin integrand $f\colon T\times X \to \Rex$ with $f_t\in \Gamma_0(X)$ for all $t\in T$. Let $x_0 \in X$ such that $\Expf{f}(x_0) \in \mathbb{R}$ and there exist  $k\in L^1(T)$ and $\eta>0$ such that  $f_t(x)\leq k(t)$ for all $x\in \mathbb{B}_\eta(x_0)$ and all $t\in T$. Then, for all $x\in \dom E_f$ and $\epsilon \geq 0$
      \begin{align*}
      \partial_\epsilon\Expf{f}(x) = \bigcup_{\substack{\epsilon_1,\epsilon_2\geq 0 \\ \ell\in\mathcal{Z}(\epsilon_1) \\ \epsilon_1+\epsilon_2\leq\epsilon}}\int_T \partial_{\ell(t)} f_t (x) \dmu +N_{\dom \Expf{f}}^{\epsilon_2} (x), 
      \end{align*} where $\int _T \partial_{\eta(t)} f_t(x)\dmu$ is defined by  
      $$ \left\{ \int_T \langle x^\ast(t),\,\cdot\,\rangle \dmu \in X^\ast  \colon  x^\ast\in L^1(T,X) \textrm{ and } x^\ast(t)\in\partial_{\eta(t)} f_t(x) \textrm{ for }   \mu\text{-a.e.} 
         \right\}.$$
\end{theorem}

\begin{proof}
  Consider the continuous linear map  $\Phi\colon X\to L^\infty(T,X)$ given by $\Phi(x) = \phi_x$,  where  $\phi_x\colon T\to X$ is defined by $ \phi_x(t) = x$. Then, we have that $E_f = \Intf{f}\circ \Phi$. Hence, by virtue of  Lemma \ref{lemma1}, $\Intf{f}$ is continuous at $\phi_{x_0}$. Moreover, by the Chain Rule for the $\epsilon$-subdifferential (see, e.g. \cite[Theorem 2.8.1]{MR1921556}), we get that  
     \begin{align*}
     \partial_\epsilon E_f(x) = \partial_\epsilon (\Intf{f}\circ \Phi)(x) = \Phi^\ast\partial_\epsilon \Intf{f}(\Phi (x)) \textrm{ for all } x\in \dom E_f.
     \end{align*}
    Moreover, if $x^\ast\in L^1(T,X)$ and $\lambda^\ast\in L^\infty(T,X)^\ast$ is a singular functional, then  
     \begin{align*}
     \Phi^\ast(x^\ast+\lambda^\ast) = \int_T \langle x^\ast(t),\,\cdot\,\rangle \dmu + \Phi^\ast(\lambda^\ast),
     \end{align*}
     where $\Phi^\ast$ is the adjoint of $\Phi$. Finally, the above equalities yield the result.
\end{proof}

The next result provides a formula for the conjungate to $E_f$. The next corollary  extends \cite[Theorem 23]{MR0373611}.

\begin{corollary}
In addition to the hypotheses of Theorem \ref{subdif_ef}, assume that  $\dom\Intf{f} = L^\infty(T,X)$.  Then,   
     \begin{align*} (E_f)^\ast(x^\ast) = \inf \{ \Intf{f^\ast}(v^\ast)\colon v^\ast\in L^1(T,X^\ast) \text{ and } \int_T v^\ast(t)\dmu = x^\ast  \}.
     \end{align*} 
\end{corollary}

\begin{proof} By  \cite[Theorem 2.8.3]{MR1921556}, we obtain that   
\begin{equation*}
    (E_f)^\ast(x^\ast) = (\Intf{f}\circ \Phi)^\ast(x^\ast)=\inf\, (\Intf{f})^\ast(v^\ast + \lambda^\ast),
\end{equation*}
where the infimum is taken over all $v^\ast\in L^1(T,X^\ast)$ and $\lambda^\ast\in L^{\text{sing}}(T,X)$ such that   $\Phi^\ast(v^\ast + \lambda^\ast) = x^\ast$. Here $\Phi$    was defined in the proof of Theorem \ref{subdif_ef}.  Finally,  since $\dom\Intf{f} = L^\infty(T,X)$, i.e., $\sigma_{\dom \Intf{f}} =\delta_{ \{0 \}}$,  we obtain the result.
\end{proof}

\section{Clarke Subdifferential of Integral Functionals}\label{sect-7}
In this section, we characterize the Clarke subdifferential of integral functionals of not necessarily convex integrands. 

\begin{theorem}\label{theo_clarke} 
 Let $X$ be  an Asplund space and consider $p\in [1,\infty)$ and $q \in (1,+\infty]$ with $\frac{1}{p}+\frac{1}{q}=1$. Let $f\colon T\times X\to \R$ be a Lusin integrand for which there exists $k\in L^q(T;\mathbb{R}_+)$ such that  $f_t $ is $k(t)$-Lipschitz for all $t\in T$. Let $ x\in L^p(T,X)$ be such that the function $(t,v)\mapsto df_t(x(t))(v)$ is a Lusin integrand, then  
    \begin{equation*}
    d\Intf{f}(x)(v) \leq \int_T  df_t(x(t))(v(t))\dmu \textrm{ for all }  v\in L^p(T,X).
    \end{equation*} Moreover, the following inclusion holds 
    \begin{eqnarray*}
        \overline{\partial}\mathcal{I}_f(x)\subset  \{ x^\ast\in L^q(T,X^\ast)\colon x^\ast(t)\in \overline{\partial} f_t(x(t)) \ \mu\text{-a.e.} \}.
    \end{eqnarray*}
\end{theorem}
\begin{proof} Since $(t,v)\mapsto df_t(x(t))(v)$ is a Lusin integrand, by Lemma \ref{lemma_lusin},  the map $t\mapsto df_t(x(t))(v(t))$ is measurable. For $y\in X$, $v\in L^p(T,X)$ and $s>0$, we define   $\varphi_{y,v,s}(t) := s^{-1}\left[ f_t(y+sv(t))-f_t(y) \right]$. Since $f_t$ is $k(t)$-Lipschitz, we obtain that $|\varphi_{y,v,s}(t)|\leq k(t)\|v(t)\|$ for all $t\in T$. Moreover, by Hölder Inequality, $k(\cdot)\|v(\cdot)\|\in L^1(T)$. Therefore, the map $t\mapsto df_t(x(t))(v(t))$ belongs to $L^1(T)$. Next, by virtue of Fatou's Lemma (see \cite[Theorem 2.8.3]{MR2267655}), for any sequence $y_n \to x$  and $s_n \to 0^+$ we have that 
\begin{eqnarray}
    \nonumber \int_T df_t(x(t))(v(t))\dmu &\geq& \int_T\limsup_{n\to\infty} \frac{f_t(y_n(t)+s_nv(t))-f_t(y_n(t))}{s_n}\dmu,
\end{eqnarray}
which implies the first inequality. To prove the inclusion, let  $x^\ast\in\overline{\partial}\mathcal{I}_f(x)$.  Then,  $x^\ast \in \partial  \mathcal{I}_g(0)$, where $g_t(u) := df_t(x(t))(u)$. Finally, by Corollary \ref{cor1}, we get that
\begin{eqnarray}
    \nonumber  x^\ast\in \partial\mathcal{I}_g(0)&\iff& x^\ast(t)\in \partial g_t(0) \ \mu\text{-a.e. } \iff  x^\ast(t)\in \overline{\partial} f_t(x(t))  \mu\text{-a.e.,} 
\end{eqnarray}
which ends the proof.
\end{proof}
 
The following result is the $L^\infty(T,X)$ counterpart of the above theorem. Its proof follows similar arguments, so it is omitted.
\begin{theorem}\label{theo_clarke2} 
 Let $X$ be  an Asplund space and   $ x\in L^\infty(T,X)$. Let $f\colon T\times X\to \R$ be a Lusin integrand for which there are $k\in L^1(T;\mathbb{R}_+)$ and $\epsilon>0$ such that 
 \begin{align*}
     | f_t(u) -f_t(v)| \leq k(t)\| u- v\|, \textrm{ for all } u,v \in \mathbb{B}_\epsilon(x(t)), \textrm{ for all } t\in T.
 \end{align*}
 Assume, in addition, that  $(t,v)\mapsto df_t(x(t))(v)$ is a Lusin integrand. Then,   
    \begin{equation*}
    d\Intf{f}(x)(v) \leq \int_T  df_t(x(t))(v(t))d\mu \textrm{ for all }  v\in L^\infty(T,X).
    \end{equation*} Moreover, the following inclusion holds 
    \begin{eqnarray*}
        \overline{\partial}\mathcal{I}_f(x)\subset  \{ x^\ast\in L^1(T,X^\ast)\colon x^\ast(t)\in \overline{\partial} f_t(x(t)) \ \mu\text{-a.e.} \}.
    \end{eqnarray*}
\end{theorem}

To apply Theorem \ref{theo_clarke}, we have to check that $(t,v)\mapsto df_t(x(t))(v)$ is a Lusin integrand.  The next proposition provides a sufficient condition ensuring this condition. We recall that for all $t\in T$, the map $(x,v)\mapsto df_t(x)(v)$ is usc (see, e.g., \cite[Chapter 2, Proposition 1.1]{MR1488695})).

\begin{proposition}
    Let $f\colon T\times X\to \R$ be a normal integrand such that $f_t$ is $k(t)$-Lipschitz, where $k$ is as in Theorem \ref{theo_clarke}. If the map $(t,x,v)\mapsto df_t(x)(v)$ is usc, then for all strongly measurable functions $x\colon T\to X$, the map $(t,v)\mapsto df_t(x(t))(v)$ is a Lusin integrand.
\end{proposition}
\begin{proof}
    Let $\epsilon >0$ and $A\subset T$ of finite measure. By Proposition \ref{lusin}, we can find a compact set $B\subset A$ with $\mu(A\setminus B)<\epsilon$ and such that $x\colon B\to X$ is continuous. Let $U\times  I$ an open set of $X\times \R$.  Then, since $(t,x,v)\mapsto df_t(x)(v)$ is usc, if $\hat{t}\in B$ is such that there exists $(\hat{v},\hat{\alpha})\in U\times I$ with $df_{\hat{t}}(x(\hat{t}))(\hat{v})\leq\hat{\alpha}$, we can find a neighborhood $V\times U_1\times U_2$ of $(\hat{t},x(\hat{t}),\hat{v})$ such that $df_{t}(x)(v)\leq\hat{\alpha}$ for all $(t,x,v)\in V\times U_1\times U_2$. Finally, if $t\in V\cap x^{-1}(U_1)$, which is a neighborhood of $\hat{t}$, then $df_t(x(t))(\hat{v})\leq \hat{\alpha}$. Hence, the set $\{ t\in B\colon (U\times I)\cap \epi df_{t}(x(t))(\cdot)\neq\emptyset \}$ is open, which ends the proof.
\end{proof}

\begin{remark} Under the  assumptions of  Theorems \ref{theo_clarke} and \ref{theo_clarke2}, except that, in this case, the map $(t,v)\mapsto df_t(x)(v)$ is a Lusin integrand for   $x\in X$, we can obtain  the following formulas for expected functionals:
\begin{equation*}
    \begin{aligned}
    dE_f(x)(v)&\leq \int_T df_t(x)(v)\dmu \textrm{ for all }  v\in X,\\
    \overline{\partial}E_f(x)&\subset \left\{\int_T x^\ast(t)\dmu\colon x^\ast\in L^1(T,X^\ast)\text{ and } x^\ast(t)\in\overline{\partial} f_t(x) \ \mu\text{-a.e.} \right\}  
    \end{aligned}
\end{equation*}
We refer to \cite{MR3783778} for similar formulas for the Clarke and limiting/Mordukhovich subdifferential of expected functionals, where the authors use \emph{weakly measurable selections} and the \emph{Gelfand integral}.   
\end{remark}

\section{Application to Calculus of Variations }\label{sect-8}

In this section, we provide optimality conditions for a Calculus of Variations problem with Lipschitz data in an Asplund space. 
Our approach is based on the Clarke penalization approach (see \cite{MR1488695}) and the estimations on the Clarke subdifferential of integral functions developed in Section \ref{sect-7} (c.f. Theorem \ref{theo_clarke}). 
The obtained result illustrates the applicability of the developed theory in previous sections.

Let $X$ be an Asplund space, and consider the problem: 
\begin{equation}\label{Problem_CV} 
  \min  J(x):=\ell(x(a),x(b)) +\int_a^b f(t,x(t), \dot{x}(t)) dt,
\end{equation}
over all admissible arcs $x\in \ACp$ satisfying $(x(a),x(b)) \in S$.  Here $S\subset X\times X$ is a nonempty and closed set and $p\in (1,+\infty)$. Moreover, we consider the following assumption:  \begin{itemize}
    \item[$(\H_{0}):$] The function $\ell\colon X\times X\to \R$ is locally Lipschitz,  $f\colon [a,b]\times X\times X \to \R$ is a normal integrand, and the maps $(x,v)\mapsto f(t,x,v)$ are $k(t)$-Lipschitz for a.e. $t\in [a,b]$ with $k\in L^q([a,b])$ and $1/p+1/q=1$.
\end{itemize}

We say that $x^{\ast}\colon [a,b]\to X$ is a \emph{local solution} of \eqref{Problem_CV} if there exists $\eps>0$ such that $J(x^{\ast})\leq J(x)$ for all admissible arcs with $\Vert x-x^{\ast}\Vert_{\ACp}\leq \eps$ and $(x^\ast(a),x^\ast(b))\in S$.

The main result of this section is the following.
\begin{theorem}\label{thm_clarke}
Let $X$ be an Asplund space and assume that $(\H_0)$ holds.  Let $x$ be a local solution of \eqref{Problem_CV}. Then, there exists $p\in \ACq$  such that 
 \begin{eqnarray}
        (\dot{p}(t),p(t)) \in  &\overline{\partial} f(t,x(t),\dot{x}(t)) \textrm{ a.e } t\in [a,b], \label{eq_euler_1} \\ (p(a),-p(b)) \in &\partial \ell(x(a),x(b)) + N_{S}(x(a),x(b)). \label{eq_euler_2}
    \end{eqnarray}
\end{theorem}

To show the above theorem, we first prove that problem \eqref{Problem_CV} is equivalent to an unconstrained one.

\begin{proposition}\label{Prop_penalizacion}
       An arc $x\colon [a,b]\to X$ is a local solution of \eqref{Problem_CV} if and only if $(x(a),x(b))\in S$ and there exists $K_0>0$ such that for $K>K_0$ $x$ is a local solution of the unconstrained problem:
        \begin{align}\label{Problem_CVP}
\hspace{-1mm}\min_{x\in \ACp} J_{K}(x):= \ell(x(a),x(b)) +\int_a^b f(t,x(t), \dot{x}(t)) dt + K d_{S}(x(a),x(b)).
\end{align}
\end{proposition}
\begin{proof}
Let $x$ be a local solution of \eqref{Problem_CV}. Since $\ell$ is locally Lipschitz, there exist $\delta>0$ and $M>0$ such that for all $(w,z), (u,v)\in \mathbb{B}_{\delta}((x(a),x(b))$
$$
\vert \ell(z,w)-\ell(u,v)\vert \leq M(\Vert z-u\Vert+\Vert w-v\Vert).
$$
Set $K_0:=M+\Vert k\Vert_1$. Proceeding by contradiction, it is possible to find $K>K_0$ and a sequence $(x_n)\subset \ACp$ such that $J_{K}(x_n)<J_{K}(x)=J(x)$ for all $n\in \mathbb{N}$, and $x_n\to x$ in $\ACp$.  For each $n\in \mathbb{N}$, let us consider $(\mu_n,\sigma_n)\in S$ such that
$$
\Vert \mu_n-x_n(a)\Vert +\Vert \sigma_n-x_n(b)\Vert \leq d_S(x_n(a),x_n(b))+\eps_n,
$$
where $\eps_n:=\min\{(2K)^{-1}(J(x)-J_K(x_n)),1/n\}$. Then, for $n\in \mathbb{N}$, we define
$$
y_n(t):=x_n(t)+\frac{t-a}{b-a}(\sigma_n-x_n(b))+\frac{b-t}{b-a}(\mu_n-x_n(a)).
$$
It is clear that $y_n\in \ACp$, $y_n(a)=\mu_n$, $y_n(b)=\sigma_n$, $(y_n(a),y_n(b))\in S$, and 
$$
\Vert y_n-x_n\Vert_{\infty}+\Vert \dot{y}_n-\dot{x}_n\Vert_{\infty}\leq \Vert \sigma_n-x_n(b)\Vert+\Vert \mu_n-x_n(a)\Vert.
$$
Moreover, let $N\in \mathbb{N}$ such that $\Vert y_n-x\Vert_{\infty}\leq \min\{\eps,\delta\}$ for all $n\in \mathbb{N}$. Hence, 
\begin{equation*}
    \begin{aligned}
    J(y_n)-J(x_n)&= \ell(y_{n}(a),y_{n}(b)) - \ell(x_n(a),x_n(b))\\ 
 &+\int_a^b [f(t,y_{n}(t),\dot{y}_{n}(t))-f(t,x_n(t),\dot{x}_n(t))]dt\\
&  \leq  M(\|\sigma_n-x_n(b)\| + \|\mu_n-x_n(a)\|)\\ 
  &+ \int_a^b k(t)(\|y_{n}(t)-x_n(t)\| + \|\dot{y}_{n}(t)-\dot{x}_n(t)\|)dt\\
  & \leq M(\|\sigma_n-x_n(b)\| + \|\mu_n-x_n(a)\|) \\
  & + (\|\sigma_n-x_n(b)\|+\|\mu_n-x_n(a)\|)\int_a^b k(t)dt\\
&  \leq (\|\sigma_n-x_n(b)\|+\|\mu_n-x_n(a)\|)\left(M+\|k\|_1\right),
    \end{aligned}
\end{equation*}
which shows that 
$$
J(y_n)\leq J(x_n)+K_0(\Vert \mu_n-x_n(a)\Vert+\Vert \sigma_n-x_n(b)\Vert) \textrm{ for all } n\geq N.
$$
Moreover, we observe that $(y_n)$ is feasible for \eqref{Problem_CV}. Hence, 
$$
J(x)\leq J(x_n)+K_0(\Vert \mu_n-x_n(a)\Vert+\Vert \sigma_n-x_n(b)\Vert) \textrm{ for all } n\geq N.
$$
Therefore, 
$$
J(x)\leq J(x_n)+K_0(d_S(x_n(a),x_n(b))+\eps_n)\leq J_K(x_n)+\frac{1}{2}(J(x)-J_k(x_n))<J(x),
$$
which is a contradiction.  Finally, the converse implication is direct.

\end{proof}

Now, let us consider the closed and convex set 
$$A:= \left\{ (u,x,y)\in  \mathcal{X}\colon  x(t) = u+\int_a^ty(s)ds \textrm{ for all } t\in [a,b]\right\},$$
where $\mathcal{X}=X\times L^p([a,b],X)\times L^p([a,b],X)$.  It is clear that if $x$ is a local solution of \eqref{Problem_CVP}, then $(x(a),x,\dot{x})$ is a local solution of the following problem: 
\begin{equation}\label{Problem_CV2}
  \min_{(y,x,y)\in A}  J_K(x)= \ell(x(a),x(b)) +\int_a^b f(t,x(t),y(t))dt + K\cdot d_{S}(x(a),x(b)).
\end{equation}
The next result, known as Dubois-Reymond Lemma, can be proved in a similar way to \cite[Proposition~3.5.21]{MR1488695}.
\begin{lemma}\label{lema_dub} Let $(u,x,y)\in A$. If $(u^\ast,x^\ast,y^\ast)\in N_A(u,x,y)$, then $$y^\ast(t)=-\int_t^b x^\ast(s)ds \text{ for all } t\in [a,b] \textrm{ and } y^\ast(a) = u^\ast = - \int_a^b x^\ast(s)ds.
$$
\end{lemma}

Now, we are in position to prove Theorem \ref{thm_clarke}.  We recall that $L^p([a,b],X)$ is an Asplund space (see, e.g.,  \cite[Corollary 3, p. 100]{MR0453964}).  \\
\noindent \emph{Proof of Theorem \ref{thm_clarke}:} Let us consider: 
$$
\Phi(u,x,y):=(u,u+\int_a^b y(s)ds) \textrm{ and } \mathcal{I}_f(u,x,y):=\int_a^b f(t,x(t),y(t))dt.
$$
Hence, it is clear that problem \eqref{Problem_CV2} can be written as 
    \begin{align}\label{Problem_CV3}
  \min\,  (\ell\circ\Phi)(u,x,y) +\Intf{f}(u,x,y) + (K\cdot d_{S}\circ \Phi)(u,x,y) + \delta_A(u,x,y),
\end{align}
over $(u,x,y)\in \mathcal{X}$.

Let $x$ be a local solution of (\ref{Problem_CV}).  Then, by Proposition \ref{Prop_penalizacion}, it is clear that  $z:=(x(a),x,\dot{x})$ is a local solution of (\ref{Problem_CV3}). Hence, by the optimality condition, the sum and the chain rule (see, e.g.,  \cite[Chapter~6]{MR2986672}), it follows that
\begin{equation*}
\begin{aligned}
(0,0,0) &\in \partial(\ell\circ\Phi +\Intf{f} + K\cdot d_{S}\circ \Phi+\delta_A)(z)\\
&\subset  \partial(\ell\circ\Phi)(z) +\partial\Intf{f}(z) + \partial((K\cdot d_{S})\circ \Phi)(z)+\partial\delta_A(z)\\
&\subset \Phi^\ast \partial\ell(\Phi(z))+\partial\Intf{f}(z)
+\Phi^\ast(K\partial d_S(\Phi(z)))+N_A(z)\\
&\subset \Phi^\ast \partial\ell(\Phi(z)) + \partial \Intf{f}(z) + \Phi^\ast N_{S}(\Phi(z)) + N_A(z),
    \end{aligned}
\end{equation*}
where we have used that $\Phi(z)=(x(a),x(b))\in S$ and  $\Phi^\ast(K\partial d_S(\Phi(z)))\subset \Phi^\ast N_S(\Phi(z))$. Therefore, $(0,0,0)\in \Phi^\ast \partial\ell(\Phi(z)) + \overline{\partial} \Intf{f}(z) + \Phi^\ast N_{S}(\Phi(z)) + N_A(z)$. \newline
The above condition implies the existence of $(u^\ast_1,v^\ast)\in \partial\ell(\Phi(z))$, $(u^\ast_2,x^\ast,y^\ast)\in N_A(z)$, and $(u_3^\ast,w^\ast)\in N_{S}(\Phi(z))$ such that
$$
-\Phi^\ast(u^\ast_1,v^\ast) - (u^\ast_2,x^\ast,y^\ast)-\Phi^\ast(u^\ast_3,w^\ast)\in \overline{\partial}\Intf{f}(z).
$$
Moreover, since $\Phi^\ast(u^\ast,v^\ast) = (u^\ast+v^\ast,0,v^\ast)$ and $\overline{\partial} \Intf{f}(z) = \{0\}\times \overline{\partial}\Intf{f}(x,\dot{x})$, we obtain that $ u_1^\ast + u_2^\ast + u_3^\ast + v^\ast + w^\ast = 0$  and
$$
(-x^\ast,-v^\ast-w^\ast-y^\ast)\in \overline{\partial}\Intf{f}(x,\dot{x}).
$$
Next, by virtue of Lemma \ref{lema_dub}, $y^\ast(t) = -\int_t^b x^\ast(s)ds$ and $u_2^\ast =- \int_a^b x^\ast(s)ds$.  Hence, by taking $p(t) := u_1^\ast + u_2^\ast +u_3^\ast +  \int_t^bx^\ast(s)ds$, we get that $\dot{p}(t)=-x^{\ast}(t)$ and 
$$
p(t) = u_1^\ast + u_2^\ast +u_3^\ast+ \int_t^bx^\ast(s)ds= -v^\ast - w^\ast -y^\ast.
$$
We conclude that $(\dot{p},p)\in \overline{\partial}\Intf{f}(x,\dot{x})$.  Hence, by Theorem \ref{theo_clarke}, we obtain \eqref{eq_euler_1}. Finally, due to previous calculations, we get that
\begin{equation*}
    \begin{aligned}
(p(a),-p(b)) &= (u_1^\ast + u_2^\ast +u_3^\ast+ \int_a^bx^\ast(s)ds,-u_1^\ast - u_2^\ast-u_3^\ast )\\
&= (u_1^\ast + u_2^\ast + u_3^\ast+ \int_a^bx^\ast(s)ds,v^\ast+w^\ast )\\
&= (u_1^\ast,v^\ast) + (u_3^\ast,w^\ast)\\
 &\in \partial\ell(\Phi(z)) + N_{S}(\Phi(z))\\
 &=  \partial\ell(x(a),x(b)) + N_{S}(x(a),x(b)),
    \end{aligned}
\end{equation*}
which proves \eqref{eq_euler_2}. The proof of Theorem \ref{thm_clarke} is completed.
\qed

\section{Application to Sweeping Processes}\label{sect-9}

Let $\H$ be a Hilbert space. The sweeping process is a first-order differential inclusion involving the normal cone to a moving closed convex set. It was introduced by J.J. Moreau in the early seventies as a model for elastoplasticity in contact mechanics (see \cite{MO1,MO2,MO4}).  In its simplest form, it consists of  finding an absolutely continuous solution $x\colon [a,b]\to \H$ of the following differential inclusion:
\begin{equation}\label{SSPP}
\left\{
    \begin{aligned}
    \dot{x}(t)&\in -N_{C(t)}(x(t)) \textrm{ a.e. } t\in [a,b],\\
    x(0)&=x_0\in C(0).
    \end{aligned}
    \right.
\end{equation}
However, due to the presence of the normal cone, the above definition requires implicitly that $x(t)\in C(t)$ for all $t\in [a,b]$. Hence, so that the sweeping process admits a solution in the above sense, we require some continuity properties on the set-valued map $C$ (jumps on the moving sets are not allowed). Indeed, the first result on the existence of solutions for the sweeping process \eqref{SSPP} (see \cite{MO1}) requires that
\begin{align}\label{hyp_haus}
\operatorname{Haus}(C(t),C(s))\leq \kappa \vert t-s\vert  \quad t,s\in [a,b],
\end{align}
where $\kappa \geq 0$ and $\operatorname{Haus}(A,B)$ is the Hausdorff distance between $A$ and $B$.

Due to the practical interest in contact mechanics problems, where collisions and jumps can occur (see, e.g., \cite{MR1710456,MR3467591,MR2857428}), efforts have been made to extend the notion of  solution to the case of discontinuous moving sets (typically of bounded variation). In this section, we study two notions of solutions in the discontinuous setting and we prove that they are equivalent in a very general framework.

The first definition is based on \cite{MR3571564} and is the most widely used notion of solution.
\begin{definition}\label{def2}
We say that $x\colon [a,b]\to \H$ is a \emph{solution in the sense of differential measures} of the sweeping process \eqref{SSPP}  if 
\begin{enumerate}[label=\textit{(\alph*)}, ref=\ref{def2}-(\alph*)]
    \item The mapping $x(\cdot)$ is of bounded variation on $[a,b]$, right continuous, and satisfies $x(a)=x_0$ and $x(t)\in C(t)$ for all $t\in [a,b]$.
    \item  There exists a positive Radon measure $\nu$ with respect to which the differential measure $dx$ of $x(\cdot)$ is absolutely continuous with $\frac{dx}{\dnu}(\cdot)$ as an $L_{\nu}^1([a,b],\H)$-density and 
$$
\frac{dx}{\dnu}(t)\in -N_{C(t)}(x(t)) \quad \nu\textrm{-a.e. } t\in [a,b].
$$
\end{enumerate}
\end{definition}
Now, we consider the notion of  ``\emph{integral solution}'', which is based on  \cite{MR2774131,MR2491851} (see also  \cite{MR3369277}).
\begin{definition}\label{def1}
We say that $x\colon [a,b]\to \H$ is an \emph{integral solution} of the sweeping process \eqref{SSPP} if
\begin{enumerate}[label=\textit{(\alph*)}, ref=\ref{def1}-(\alph*)]
    \item The mapping $x(\cdot)$ is of bounded variation on $[a,b]$, right continuous, and satisfies $x(a)=x_0$ and $x(t)\in C(t)$ for all $t\in [a,b]$.
    \item There exists a positive Radon measure $\nu$ with respect to which the differential measure $dx$ of $x(\cdot)$ is absolutely continuous with $\frac{dx}{\dnu}(\cdot)$ as an $L_{\nu}^1([a,b],\H)$-density and for all $y\in \mathcal{C}([a,b],\H)$ with  $y(t)\in C(t)$ $\nu$-a.e. $t\in [a,b]$:
\begin{equation*}
    \int_a^b\left\langle \frac{dx}{\dnu}(t),y(t)-x(t)\right\rangle \dnu \geq 0.
\end{equation*}
\end{enumerate}

\end{definition}

We observe  that the equivalence between both definitions is not trivial at all. Indeed, on the one hand, Definition \ref{def1} requires the existence of continuous test functions with values in $C$, which do not necessarily exist. On the other hand, Definition \ref{def2} is based on a punctual characterization of the normal cone.  In any case, a good characterization of solution  may be appropriate for different aspects of sweeping processes.

The next result shows, as a consequence of Theorem \ref{mainresult}, the equivalence between these two definitions for a general set-valued map $C\colon [a,b]\tto \H$ with nonempty, closed, convex and bounded values. Here, we take $T = [a,b]$ and $\mathcal{A}$ as the Lebesgue $\sigma$-algebra, denoted by $\mathcal{L}([a,b])$.

\begin{theorem}
Assume that  $C\colon [a,b]\tto \H$ is a set-valued map with nonempty, closed, convex and bounded values. Then, $x(\cdot)$ is an integral solution if and only if $x(\cdot)$ is a solution  in the sense of differential measures. 
\end{theorem}
\begin{proof}
On the one hand, if $x$ is a solution in the sense of differential measures (see Definition \ref{def2}), it is enough to integrate the function $t\mapsto \langle\frac{dx}{\dnu}(t),y(t)-x(t)\rangle$ for all test function $y\in \mathcal{C}([a,b],\H)$ to conclude that $x$ is an integral solution.\\
On the other hand, let $x\colon [a,b]\to \H$ be an integral solution, i.e., for all $y\in \mathcal{C}([a,b],\H)$ with  $y(t)\in C(t)$ $\nu$-a.e. $t\in [a,b]$: $$\int_a^b\left\langle \frac{dx}{\dnu}(t),y(t)-x(t)\right\rangle \dnu \geq 0.$$
Then, we have that 
$$ \beta := \inf_{y\in \mathcal{C}([a,b],\H)} \int_a^b \left\langle \frac{dx}{\dnu}(t),y(t)-x(t)\right\rangle + \delta_{C(t)}(y(t))\dnu\geq 0.$$ Now, let us define $f_t(y):=\left\langle \frac{dx}{\dnu}(t),y-x(t)\right\rangle + \delta_{C(t)}(y)$.  It is possible to prove that $f$ is a Lusin integrand. Indeed, let $A\in\mathcal{L}([a,b])$ and $\epsilon>0$, we can choose a compact set $K\subset A$ with $\nu(A\setminus K)<\epsilon$ such that $\frac{dx}{\dnu}\colon K\to \H$ and $x\colon K\to \H$ are continuous (Lusin Theorem). Then, $K\times X\ni(t,y)\mapsto \left\langle \frac{dx}{\dnu}(t),y-x(t)\right\rangle$ is continuous, so it is a Scorza-Dragoni function. Also, we can observe that for all $t\in [a,b]$, $\epi \delta_{C(t)} = C(t)\times [0,+\infty)$ and given that $C$ is lsc (see \cite[Proposition 6.1.37]{MR2527754}), we have that $t\tto \epi \delta_{C(t)}$ is lsc, then by Proposition \ref{binary_prop}, $K\ni t\tto \epi f_t$ is lsc. By Theorem \ref{mainresult}, there exists a nondecreasing sequence of compacts sets $(K_n)_{n\in\mathbb{N}}$ such that 
\begin{align*}
		\beta_n:=\inf_{\phi\in {C}(K_n,\H)}\int_{K_n} f_t(\phi(t)) \dnu =\int_{K_n}\inf_{y\in \H}f_t(y)\dnu,
\end{align*}
with $\nu([a,b]\setminus (\bigcup K_n))=0$. For all $n\in\mathbb{N}$, there exists $\phi_n\in {C}(K_n,\H)$ such that $$ \beta_n+\frac{1}{n}\geq \int_{K_n}f_t(\phi_n(t))\dnu,$$
which implies that $\phi_n(t)\in C(t) \ \nu$-a.e. $t\in K_n$. The hypothesis about the continuity under Hausdorff distance of $C$ allows us to obtain that $\phi_n(t)\in C(t)$ for all $t\in K$. Let us consider $\hat M_n(t)=U_n(t)$ if $t\in K_n$ and $M_n(t)=\H$ if  $t\notin K_n$, where $U_n(t) = \{x\in H:\|\phi_n(t)-x\|<1/n\}$. It is not difficult to see that $\hat M_n$ has open graph and  $\hat M_n(t)\cap C(t)\neq \emptyset$ for all $t\in [a,b]$. Then, by virtue of \cite[Proposition 6.1.24]{MR2527754}, the multifunction $\hat M_n\cap C$ is lsc. Hence,  due to  \cite[Proposition 6.1.19]{MR2527754}, we have that $M_n(t) := \cl(\hat M_n(t)\cap C(t))$ is lsc and
\[   
M_{n}\colon t\rightrightarrows 
     \begin{cases}
       \cl\left({U_n(t)\cap C(t)}\right) &\quad\text{ if } t\in K_n, \\
       C(t) &\quad\text{ if } t\notin K_n. \\
     \end{cases}
\]
It is clear that $M_n$ takes closed, convex and nonempty values. Hence, by Michael's Selection Theorem, we obtain a continuous selection $\varphi_n$ of $M_n$ such that $\|\varphi_n(t)-\phi_n(t)\|\leq 1/n$ for all $t\in K_n$. Now, we observe that 
\begin{eqnarray*}
    \int_{K_n}f_t(\phi_n(t))\dnu &=& \int_a^b f_t(\varphi_n(t))\dnu+\int_{K_n} f_t(\phi_n(t))-f_t(\varphi_n(t))\dnu \\
    &&- \int_{[a,b]\setminus K_n}f_t(\varphi_n(t))\dnu\\
    &\geq & \beta + \int_{K_n} f_t(\phi_n(t))-f_t(\varphi_n(t))\dnu - \int_{[a,b]\setminus K_n}f_t(\varphi_n(t))\dnu.
\end{eqnarray*}
Moreover, on the one hand,
\begin{equation*}
    \begin{aligned}
      \left|\int_{K_n} f_t(\phi_n(t))-f_t(\varphi_n(t))\dnu\right|&\leq \int_{K_n}|f_t(\phi_n(t))-f_t(\varphi_n(t))|\dnu\\
    &\hspace{-12mm}=\int_{K_n}\left|\left\langle \frac{dx}{\dnu}(t),\phi_n(t)-\varphi_n(t)\right\rangle\right|\dnu \leq\frac{1}{n} \int_a^b \left\Vert \frac{dx}{\dnu}(t)\right\Vert \dnu. 
    \end{aligned}
\end{equation*}
On the other hand,
\begin{equation*}
    \begin{aligned}
   \left|\int_{[a,b]\setminus K_n}f_t(\varphi_n(t))\dnu\right| 
    &\leq \int_{[a,b]\setminus K_n}\left\|\frac{dx}{\dnu}(t)\right\|\cdot \|\varphi_n(t)-x(t)\|\dnu\\
    &\leq R \int_{[a,b]\setminus K_n}\left\|\frac{dx}{\dnu}(t)\right\|\dnu, 
    \end{aligned}
\end{equation*}
where $R>0$ satisfies that $\sup_{t\in [a,b]}\text{diam}(C(t))\leq R$, such $R$ exists because $C$ has bounded values, and the assumption \eqref{hyp_haus} allows us to prove that $C$ has uniform bounded values. Therefore, for all $n\in\mathbb{N}$, 
$$ \beta_n + \frac{1}{n}\geq \beta - \frac{1}{n} \int_a^b \left\Vert \frac{dx}{\dnu}(t)\right\Vert \dnu -R\int_{[a,b]\setminus K_n}\left\|\frac{dx}{\dnu}(t)\right\|\dnu.$$ 
Hence, taking $n\to \infty$, we get that $\liminf_{n\to\infty} \beta_n\geq \beta$. Moreover,  
$$ \hspace{-1mm}\liminf_{n\to\infty}\beta_n = \liminf_{n\to\infty}\int_{K_n}\inf_{y\in \H}f_t(y)\dnu = \lim_{n\to\infty}\int_{K_n}\inf_{y\in \H}f_t(y)\dnu= \int_a^b\inf_{y\in \H}f_t(y)\dnu, $$  since  $t\mapsto\inf_{y\in \H}f_t(y)$ is integrable.  Therefore,  $\int_a^b\inf_{y\in \H}f(t,y)\dnu\geq \beta\geq 0$. Finally, since $\inf_{y\in \H}f(t,y)\leq 0$ $\nu$-a.e. $t\in [a,b]$ (because $x(t)\in C(t)$ $\nu$-a.e. $t\in [a,b]$), we conclude that $\inf_{y\in \H} f_t(y)\geq 0$ $\nu$-a.e $t\in [a,b]$, which proves that $x$ is a solution in the sense of differential measures.
\end{proof}

\bibliographystyle{plain}
\bibliography{references}
\end{document}